\documentclass[11pt]{article}
\usepackage[T1]{fontenc}
\usepackage{lmodern}
\usepackage[a4paper,margin=1in]{geometry}
\usepackage{amsmath,amssymb,amsthm,mathtools}
\usepackage{authblk}
\usepackage{microtype}
\usepackage{array}
\usepackage{enumitem}
\usepackage[colorlinks=true,linkcolor=blue,citecolor=blue,urlcolor=blue]{hyperref}
\hypersetup{
 pdftitle={On Erd\H{o}s--Ko--Rado and Hilton--Milner Theorems for Direct Products},
 pdfauthor={Tian Yao, Mengyu Cao and Kaishun Wang},
 pdfsubject={Linear EKR thresholds and a multipart t-Hilton--Milner classification},
 pdfkeywords={Erdos-Ko-Rado theorem, Hilton-Milner theorem, direct product, t-intersecting family, generating set, shifting, intersection shadow}
}

\newtheorem{theorem}{Theorem}[section]
\newtheorem{lemma}[theorem]{Lemma}
\newtheorem{conjecture}[theorem]{Conjecture}
\newtheorem{proposition}[theorem]{Proposition}
\newtheorem{corollary}[theorem]{Corollary}
\newtheorem{problem}[theorem]{Problem}
\newtheorem*{claim}{Claim}
\theoremstyle{definition}

\theoremstyle{remark}

\newcommand{\cA}{\mathcal A}
\newcommand{\cB}{\mathcal B}
\newcommand{\cC}{\mathcal C}
\newcommand{\cD}{\mathcal D}
\newcommand{\cE}{\mathcal E}
\newcommand{\cF}{\mathcal F}
\newcommand{\cG}{\mathcal G}
\newcommand{\cH}{\mathcal H}
\newcommand{\cI}{\mathcal I}
\newcommand{\cJ}{\mathcal J}

\newcommand{\cP}{\mathcal P}
\newcommand{\cQ}{\mathcal Q}
\newcommand{\cR}{\mathcal R}
\newcommand{\cT}{\mathcal T}
\newcommand{\cU}{\mathcal U}
\newcommand{\cV}{\mathcal V}

\newcommand{\bs}{\boldsymbol}

\title{On Erd\H{o}s--Ko--Rado and Hilton--Milner Theorems for Direct Products}
\author[1]{Tian Yao\thanks{E-mail: \texttt{tyao@hist.edu.cn}. Supported by Natural Science Foundation of Henan (262300422621).}}
\author[2]{Mengyu Cao\thanks{Corresponding author. E-mail: \texttt{myucao@ruc.edu.cn}. Supported by the National Natural Science Foundation of China (12301431) and Beijing Natural Science Foundation (1262010).}}
\author[3]{Kaishun Wang\thanks{E-mail: \texttt{wangks@bnu.edu.cn}. Supported by the National Natural Science Foundation of China (12131011, 12571347) and Beijing Natural Science Foundation (1252010, 1262010).}}
\affil[1]{School of Mathematical Sciences, Henan Institute of Science and Technology, Xinxiang 453003, China}
\affil[2]{Institute for Mathematical Sciences, Renmin University of China, Beijing 100086, China}
\affil[3]{Laboratory of Mathematics and Complex Systems (Ministry of Education), School of
	Mathematical Sciences, Beijing Normal University, Beijing 100875, China}
\date{}

\begin{document}
\maketitle

\begin{abstract}
	We investigate $t$-intersecting families in direct-product set systems obtained by
	prescribing the number of selected elements in each part of a partitioned
	ground set. 
	For a finite union of layers, we prove an Erd\H{o}s--Ko--Rado result under coordinatewise linear part-size conditions. 
	As an application, we establish a new range of parameters for which a conjecture of Frankl et al.\ [\emph{J. Combin. Theory Ser. A} \textbf{155} (2018), 493--502] holds. 
	Under an explicit polynomial large-part hypothesis,  we also characterize the maximum nontrivial $t$-intersecting families for the single-layer setting.  In particular, for $t=1$, this answers the problem
	of Kwan et al.\ [\emph{J. Combin. Theory Ser. A} \textbf{156} (2018), 44--60] asking for a
	classification of all extremal families, including the possible
	non-shifted maximizers.
\end{abstract}


\section{Introduction}

For a positive integer $m$, write $[m]=\{1,\ldots,m\}$ and set $\binom m\ell=0$ whenever $\ell<0$ or $\ell>m$.  
For a finite set $Y$ and an integer $r\ge0$, write $\binom Yr$ for the collection of all $r$-subsets of $Y$.  
A family $\cF\subseteq2^Y$ is \emph{$t$-intersecting}, where $t$ is omitted when $t=1$, if $|F\cap F'|\ge t$ for any $F,F'\in\cF$. 
It is \emph{trivial} if all its members contain a common $t$-set.  
The members of a fixed ambient family containing a given $t$-set form a \emph{full $t$-star}.

The Erd\H{o}s--Ko--Rado theorem \cite{EKR} initiated the extremal theory
of intersecting families.  Frankl \cite{FranklEKR} and Wilson \cite{Wilson} determined the sharp
$t$-star threshold.  
The prototype for nontrivial $t$-intersecting families is the Hilton--Milner theorem \cite{FranklHM,HiltonMilner}. 
Ahlswede and Khachatrian \cite{AKnontrivial,AKcomplete} subsequently solved the complete intersection
problem and its nontrivial analogue.  Further
structural results for large nontrivial $t$-intersecting families appear
in \cite{CLW}.

In this paper, we study the intersection problem in the direct product setting.
Let $X$ be a finite set, and
\[
X=X_1\mathbin{\dot\cup}\cdots\mathbin{\dot\cup}X_p,
\qquad |X_i|=n_i,
\]
where the blocks $X_i$ are nonempty. For $\bs k=(k_1,\ldots,k_p)$ with $0\le k_i\le n_i$ for every $i\in[p]$, put
\begin{equation*}
	\cH(\bs k)=
	\{F\subseteq X:|F\cap X_i|=k_i\text{ for every }i\in[p]\}.
\end{equation*}
Katona \cite{KatonaProduct} introduced a multilayer setting. Let  $R\subseteq\prod_{i=1}^p([n_i]\cup\{0\})$ be nonempty, and write
\begin{equation}\label{eq:union-layers}
	\cH_R=\bigcup_{\bs r\in R}\cH(\bs r),\qquad
	b_i=\max_{(r_1,\dots,r_p)\in R}r_i.
\end{equation}
Without loss of generality, we always assume that $b_i>0$ for every $i\in[p]$. 
In particular, when a problem involves only one fixed vector $\bs k$, we  assume that
all coordinates of $\bs k$ are positive.

\subsection{Erd\H{o}s--Ko--Rado results for direct products}

For the case $R=\{\bs k\}$, Frankl \cite{FranklProduct} determined the maximum size of intersecting subfamilies of $\cH(\bs k)$. 
In \cite{YLW}, Yao, Lv and Wang proved that every maximum
$t$-intersecting subfamily of $\cH(\bs k)$ is a full $t$-star when
$
n_i>2(t+1)pk_i^2$ for each $i\in[p]$. 
We also refer the reader to \cite{AAK}
for a local Ahlswede--Khachatrian type result. 

When $|R|\ge2$,  Katona \cite{KatonaProduct} treated the case $p=2$ and determined the maximum size of intersecting subfamilies of $\cH_R$. 
In \cite{YLW}, Yao, Lv and Wang proved that, for $R\subseteq\prod_{i=1}^p[n_i]$ and under an additional restriction on $t$, every maximum $t$-intersecting subfamily of $\cH_R$ is a full $t$-star provided that the part sizes satisfy certain explicit nonlinear lower bounds.

Our first contribution is a coordinatewise linear threshold under which every maximum $t$-intersecting subfamily of $\cH_R$ is a full $t$-star.

\begin{theorem}\label{thm:multi-layer}
	Let $p,t\ge 1$ and  $\max_{\bs r\in R}\sum_i r_i\ge t$. If $n_i\ge3ptb_i$ for each $i\in[p]$, 
	then every $t$-intersecting family $\cF\subseteq\cH_R$ satisfies
	\[
	|\cF|\le\max_{\substack{s_1+\cdots+s_p=t\\(s_1,\dots,s_p)\in\mathbb{Z}_{\ge0}^p}}\sum_{\bs r\in R}
	\prod_{i=1}^p\binom{n_i-s_i}{r_i-s_i}.
	\]
	Equality holds if and only if $\cF$ is a full $t$-star attaining the
	maximum on the right-hand side.
\end{theorem}

In the special case $R=\{\bs k\}$, $p=1$ and $k_1\ge t+1$, the sharp classical threshold  for the full-$t$-star conclusion is
$n_1>(t+1)(k_1-t+1)$.
Consequently, Theorem~\ref{thm:multi-layer} is principally intended for $(|R|,p)\neq(1,1)$.  It remains open to determine the sharp part-size conditions guaranteeing
the full-star conclusion of Theorem~\ref{thm:multi-layer}, or, more generally,
to establish a corresponding complete intersection theorem.

Frankl, Han, Huang and Zhao \cite{FHHZ} posed the following conjecture, and pointed out that the conditions stated are necessary. Wagner \cite{Wagner} later showed that these necessary conditions are
not sufficient in general. Nevertheless, Theorem~\ref{thm:multi-layer} establishes a 
range of parameters in which its conclusion remains valid.

\begin{conjecture}[{\cite[Conjecture~12]{FHHZ}}]
	Let $p\ge1$ and  $n_i>a_i\ge0$ for each $i\in[p]$. Suppose 
	$a=\sum_i a_i$ and $|X|\ge k\ge1$ with $k\ge a$.  Define
	\[
	\cG(\bs a,k)=
	\left\{F\in\binom Xk:|F\cap X_i|\ge a_i
	\text{ for every }i\in[p]\right\}.
	\]
	If $n_i\ge2a_i$ for every $i$ and
	$n_i>k-a+a_i$ for all but at most one index with $a_i>0$, then every
	maximum intersecting subfamily of $\cG(\bs a,k)$ is a full $1$-star.
\end{conjecture}

To derive the promised range, let $R$ consist of the vectors
$\bs r\in\mathbb Z_{\ge0}^p$ with $\sum_i r_i=k$ and $n_i\ge r_i\ge a_i$. 
Then $\cG(\bs a,k)=\cH_R$. By $k-a+a_i\ge r_i$, we have $k-a+a_i\ge b_i$ and obtain the following corollary.

\begin{corollary}
	If
	$
	n_i\ge 3p(k-a+a_i)$ for $i\in[p]$, 
	then every maximum intersecting subfamily of
	$\cG(\bs a,k)$ is a full $1$-star.
\end{corollary}

\subsection{Hilton--Milner results for direct products}\label{HMR}

We next turn to the Hilton--Milner problem in the single-layer setting.  
Kwan, Sudakov and Vieira \cite{KSV} investigated the nontrivial intersecting families 
when  the part sizes are sufficiently large, using the shifting
technique, and determined the maximum size of these families but did not classify the
original, non-shifted maximizers. They also posed the problem of characterizing the extremal families; see  \cite[Section 4]{KSV}. We formulate this problem as follows.
\begin{problem}\label{prob:KSV}
	For fixed $\bs k$ and sufficiently large part sizes
	$n_1,\ldots,n_p$, characterize all maximum nontrivial intersecting
	subfamilies of $\cH(\bs k)$, without assuming that the family is
	shifted.
\end{problem}

A natural extension is the corresponding nontrivial intersection problem
for $t$-intersecting families. Recent progress in this direction 
mainly concerns the transversal specialization
$k_1=\cdots=k_p=1$. 
Frankl and Nie \cite{FN} solved the problem for arbitrary part sizes when $p=t+2$, 
and for equal part sizes when $p\ge t+2$ and the common part size is sufficiently large. 
Hou and Hu \cite{HouHu} subsequently solved the symmetric transversal problem 
with $1\le t\le p-2$ and the common part size at least $2$. 
For asymmetric Cartesian products,
they reduced the case $t=1$ to a finite weighted downset problem and
obtained explicit formulas for $p=4,5,6$.

Our second contribution in this paper is Theorem~\ref{thm:closed-HM}, which, under an explicit large-part hypothesis, determines the maximum size of a nontrivial $t$-intersecting subfamily of $\cH(\bs k)$ for arbitrary choices of $\bs k$, possibly unequal part sizes, and general $t$, and classifies all families attaining this maximum.
In particular, when $t=1$, it answers Problem~\ref{prob:KSV}.

To prepare for our main result, we first introduce the candidate extremal families. For disjoint sets $T,Y\subseteq X$ with $|T|=t$ and $|Y|\ge2$, let
\[
	\cH^\times(T,Y)
	=\{F\in\cH(\bs k):T\subseteq F,\ F\cap Y\ne\varnothing\}
\cup\bigcup_{x\in T}
	\{F\in\cH(\bs k):F\cap T=T\setminus\{x\},\ Y\subseteq F\}.
\]
For a pair $(T,Y)$ as above,  its \emph{profile} is 
the pair
$(\bs a,\bs\ell)$ defined by
\[
a_i=|T\cap X_i|,\qquad
\ell_i=|Y\cap X_i|
\qquad(i\in[p]).
\]
Writing $\bs n=(n_1,\ldots,n_p)$, the cardinality of $\cH^\times(T,Y)$ depends only on $(\bs a,\bs\ell)$ and is given by
\[
	H^\times_{\bs n,\bs k}(\bs a,\bs\ell)
	=
	\prod_{i=1}^p\binom{n_i-a_i}{k_i-a_i}
	-\prod_{i=1}^p\binom{n_i-a_i-\ell_i}{k_i-a_i}
	+\sum_{i=1}^p a_i
	\binom{n_i-a_i-\ell_i}{k_i-a_i-\ell_i+1}
	\prod_{j\ne i}
	\binom{n_j-a_j-\ell_j}{k_j-a_j-\ell_j}.
\]

For the one-layer problem, call a set $C\subseteq X$
\emph{admissible} if it is contained in some member of
$\cH(\bs k)$; equivalently, 
$|C\cap X_i|\le k_i$ for every $i\in[p]$. 
A pair $(T,Y)$ of two disjoint subsets of $X$ is called \emph{feasible} if $T\cup\{y\}$ is admissible for every $y\in Y$
and $(T\setminus\{x\})\cup Y$ is admissible for at least one $x\in T$. 
Let $\cE_t(\bs k)$ consist of those profiles $(\bs a,\bs\ell)$ of feasible pairs $(T,Y)$ with $T\in\binom{X\setminus Y}{t}$   that satisfy one of the following mutually exclusive conditions:
\begin{enumerate}[label=\textup{(\roman*)},leftmargin=2.5em]
	\item $\sum_i\ell_i=2$;
	\item $\sum_i\ell_i\ge3$, and for every $i\in[p]$, $\ell_i\in\{0,k_i\}$ and $\ell_i\le k_i-a_i$;
	\item $\sum_i\ell_i\ge3$, and there is  a unique index $h$ with
	$1\le a_h\le k_h-1$ and $\ell_h=k_h-a_h+1$, while
	$\ell_i\in\{0,k_i-a_i\}$ for every $i\ne h$.
\end{enumerate}

Put $K=\sum_{i=1}^p k_i$ and $n_*=\min_{i\in[p]}n_i$. 
For $K\ge t+2$, define
\begin{equation*}
	\Lambda_{K,t}=\max\left\{
	3t^2(K-t+1)^2, 2(t+1)(K-t+1)^3
	\right\}.
\end{equation*}
We can now state our second main result.
\begin{theorem}
	\label{thm:closed-HM}
	Let $p,t,k_1,\dots,k_p\ge1$ and $K\ge t+2$. If  
	$
	n_*>K(1+\Lambda_{K,t})$, 
	then every nontrivial $t$-intersecting subfamily $\cF$ of $\cH(\bs k)$ satisfies
	\[
	|\cF|
	\le\max_{(\bs a,\bs\ell)\in\cE_t(\bs k)}
	H^\times_{\bs n,\bs k}(\bs a,\bs\ell).
	\]
	Equality holds if and only if $\cF=\cH^\times(T,Y)$, where $(T,Y)$ is a feasible pair whose profile is a maximizer in $\cE_t(\bs k)$.
\end{theorem}

The two boundary values $K=t$ and $K=t+1$ are elementary and are recorded
separately in Proposition~\ref{prop:HM-boundary}. Together with the theorem,
this covers every value of $K\ge t$ for which a nontrivial family may exist.

The rest of this paper is organized as follows. 
Section~\ref{sec:generators} develops the bounded-support theorem, and 
Section~\ref{sec:proofs} proves Theorem \ref{thm:multi-layer}.
Section~\ref{sec:nontrivial} proves Theorem~\ref{thm:closed-HM} and treats the boundary cases $K=t$ and $K=t+1$.

\section{Generating antichains for $t$-intersecting families}\label{sec:generators}

We now work in the ambient family $\cH_R$ defined in \eqref{eq:union-layers}. Write
\[
X_i=\{x_{i,1},\ldots,x_{i,n_i}\},\qquad
Q_i(s)=\{x_{i,j}:1\le j\le s\}\quad(i\in[p],\ 0\le s\le n_i).
\]
For $x=x_{i,a}$ and $y=x_{i,b}$ with $a<b$, define the \emph{$(x,y)$-shift}
of a family $\cF$ by replacing a member $F$ with
$(F\setminus\{y\})\cup\{x\}$ exactly when $y\in F$, $x\notin F$, and
the replacement is not already in $\cF$; all other members are left
unchanged.  Denote the resulting family by $S_{xy}(\cF)$.    
We remark that $S_{xy}(\cF)\subseteq\cH_R$, $|S_{xy}(\cF)|=|\cF|$ and  $S_{xy}(\cF)$ remains  $t$-intersecting whenever $\cF$ is $t$-intersecting. 
A family is \emph{shifted} if it is fixed by every such shift inside every part.

In the rest of this paper, for $B\subseteq X$, $\cG\subseteq2^X$ and $i\in[p]$, write
\[
B_i=B\cap X_i,\qquad
\cG[B]=\{G\in\cG:B\subseteq G\},\qquad
U_R(\cG)=\bigcup_{G\in\cG}\cH_R[G].
\]
We say that $B$ is \emph{$R$-admissible} if $\cH_R[B]\ne\varnothing$. 
A family $\cB$ of $R$-admissible sets is a \emph{generating family} for
$\cF\subseteq\cH_R$ if $\cF=U_R(\cB)$. In this case, every member of $\cB$ is called a \emph{generator} of $\cF$. Observe that if $|R|=1$, then $R$-admissibility agrees with the definition of admissibility in Section \ref{HMR}. 
For convenience, if there is no confusion, we simply say  admissible instead of $R$-admissible.

A maximum $t$-intersecting family is inclusion-maximal and consequently is an upset in the order induced on $\cH_R$, i.e., if $F\in\cF$, $F\subseteq G\in\cH_R$, then $G$ may be added without destroying $t$-intersection and therefore already
belongs to $\cF$.  Hence it has a generating antichain.

For two subsets $A,B\subseteq X$, write $A\preceq B$ if, in every block, $A_i$ can be obtained from $B_i$ by moving points to smaller labels without changing $|B_i|$.  Given a generating family $\cD$, let
\[L(\cD)=\{A:A\preceq D\text{ for some }D\in\cD\}\] and 
$L^*(\cD)$ consist of the inclusion-minimal  members of $L(\cD)$.
We call a generating antichain $\cB$ \emph{reduced left-closed} if
$\cB=L^*(\cB)$. 

The following elementary normalization of a generating family is standard for a single direct-product
layer; see \cite[Definitions 2.6--2.13 and Lemma 2.2]{AAK} for more details.  We include the
argument because here the ambient family is an arbitrary finite union of
layers.

\begin{lemma}\label{lem:reduced-left-closure}
	If $\cF\subseteq\cH_R$ is shifted and $\cD$ generates $\cF$, then
	$L^*(\cD)$ also generates $\cF$.  Moreover,
	\begin{equation}\label{eq:left-closure-property}
		A\preceq B\in L^*(\cD)
		\quad\Longrightarrow\quad
		A\in L^*(\cD)\ \text{or}\quad
		C\subsetneq A\ \text{for some }C\in L^*(\cD).
	\end{equation}
\end{lemma}

\begin{proof}
	We first prove that $U_R(L(\cD))=\cF$.  
	The inclusion
	$\cF=U_R(\cD)\subseteq U_R(L(\cD))$ is immediate.  
	For the reverse inclusion, it is enough to consider one elementary left move.  
	Suppose $D\in\cD$ and $D'=(D\setminus\{y\})\cup\{x\}$, where $x,y$ lie in the same part,
	$x$ precedes $y$, $y\in D$ and $x\notin D$.  Let
	$G\in\cH_R[D']$.  If $y\in G$, then $D\subseteq G$, so $G\in\cF$.  If
	$y\notin G$, put $G'=(G\setminus\{x\})\cup\{y\}$.  The sets $G$ and
	$G'$ lie in the same layer, $D\subseteq G'$, and $G$ is the left shift of
	$G'$.  Hence $G'\in\cF$ and shiftedness gives $G\in\cF$.  Iterating over
	a sequence of elementary left moves proves
	$\cH_R[A]\subseteq\cF$ for every $A\in L(\cD)$, implying that $U_R(L(\cD))\subseteq U_R(\cD)=\cF$.

Since every  member of  $L(\cD)$ contains an inclusion-minimal  member of $L(\cD)$, we have $U_R(L(\cD))\subseteq U_R(L^*(\cD))$. It follows from $L^*(\cD)\subseteq L(\cD)$ that $U_R(L^*(\cD))\subseteq U_R(L(\cD))$. 
Consequently, $U_R(L(\cD))=U_R(L^*(\cD))$, and we further obtain \[\cF=U_R(\cD)=U_R(L(\cD))=U_R(L^*(\cD)).\]

Finally, suppose $A\preceq B\in L^*(\cD)$.  Then $A\in L(\cD)$, and
$A$ is admissible because it has the same number of points as $B$ in each
part.  If $A$ is inclusion-minimal among the admissible members of
$L(\cD)$, then $A\in L^*(\cD)$.  Otherwise $A$ contains a proper
admissible member of $L(\cD)$, which in turn contains some
$C\in L^*(\cD)$.  This gives $C\subsetneq A$ and proves
\eqref{eq:left-closure-property}.
\end{proof}

\begin{lemma}\label{lem:generators-intersect}
	Suppose $n_i\ge2b_i$ for each $i$.  If $\cF=U_R(\cB)$ is $t$-intersecting, then $\cB$ is $t$-intersecting.
\end{lemma}
\begin{proof}
	Suppose $B,C\in\cB$ and $|B\cap C|<t$.  Choose 
	$\bs r=(r_1,\dots,r_p),\bs s=(s_1,\dots,s_p)\in R$ for which $\cH(\bs r)[B]$ and $\cH(\bs s)[C]$ are nonempty.  In block $X_i$, the fact $r_i+s_i-|B_i\cap C_i|\le2b_i\le n_i$ 
	allows us to extend $B_i$ to an $r_i$-set $F_i$ and $C_i$ to an $s_i$-set $G_i$ with
	$F_i\cap G_i=B_i\cap C_i$.  Then
	$F=\bigcup_i F_i$ and $G=\bigcup_i G_i$ belong to $\cF$ and satisfy
	$|F\cap G|=|B\cap C|<t$, a contradiction.
\end{proof}

For a generating family $\cB$, define its extent in $X_i$ by
\[
m_i(\cB)=
\max\{j:x_{i,j}\in B\text{ for some }B\in\cB\},
\]
with value $0$ if no generator meets $X_i$. 
The next lemma concerns the boundary element $x_{i,m_i(\cB)}$.

\begin{lemma}\label{lem:critical}
	Let $\cB=L^*(\cB)$ be a $t$-intersecting generating antichain of a shifted family, and put
	$m=m_i(\cB)$.  
	Suppose that $A,B\in\cB$ satisfy $x_{i,m}\in A\cap B$ and $|A\cap B|=t$.  
	Then \[
	Q_i(m)=A_i\cup B_i,\qquad m\le|A_i|+b_i-1.\]
\end{lemma}
\begin{proof}
	The definition of $m$ gives $A_i\cup B_i\subseteq Q_i(m)$.  
	For the reverse inclusion, suppose for contradiction that $x_{i,j}\in Q_i(m-1)\setminus(A_i\cup B_i)$.  
	Write $\widetilde A=(A\setminus\{x_{i,m}\})\cup\{x_{i,j}\}$. 
	By \eqref{eq:left-closure-property}, $\widetilde A$  contains some $C\in\cB$.  
	Then $|B\cap C|\le |B\cap\widetilde A|=t-1$, a contradiction to the $t$-intersection property of $\cB$.  
	Thus $A_i\cup B_i=Q_i(m)$.
	
	We have $|A_i\cap B_i|\ge1$ by $x_{i,m}\in A\cap B$, and $|B_i|\le b_i$ since $B$ is admissible. 
	Then 
	\[
	m=|A_i\cup B_i|=|A_i|+|B_i|-|A_i\cap B_i|\le |A_i|+b_i-1,
	\]
	as desired.
\end{proof}

Let $\cF=U_R(\cB)$ be shifted and $t$-intersecting, where $\cB=L^*(\cB)$ is a generating antichain. 
Fix $i\in[p]$ with $m=m_i(\cB)>0$. Set
\[
\cB^0=\{B\in\cB:x_{i,m}\notin B\},\qquad
\cB^1=\{B\in\cB:x_{i,m}\in B\}.
\]
Fix a total order on $\cB$ in which every member of $\cB^0$  precedes every member of $\cB^1$, and assign each $F\in\cF$ to the first generator contained in $F$.
For $A\in\cB^1$, let $w(A)$ be the number of sets assigned to $A$, and set
\[
W=\sum_{A\in\cB^1}w(A).
\]

\begin{lemma}\label{lem:ledger}
	In the preceding setup, assume $n_j\ge2b_j$ for every $j\in[p]$ and
	$n_i>2b_i$.  Suppose also that, for every $A\in\cB^1$, 
	its size is at least $t+1$ and
	$(\cB\setminus\{A\})\cup\{A^-\}$, where 
	$A^-=A\setminus\{x_{i,m}\}$, is not $t$-intersecting.  
	If $\cA\subseteq\cB^1$ and no two distinct members of
	$\cA$ have intersection exactly $t$, then 
	$
	\cB'=\cB^0\cup\{A^-:A\in\cA\}
	$ 
	is $t$-intersecting and
	\[
		|U_R(\cB')|
		\ge |\cF|-W+\left(\frac{n_i}{b_i}-1+\frac1{b_i}\right)\sum_{A\in\cA}w(A).
	\]
\end{lemma}

\begin{proof}
	Lemma~\ref{lem:generators-intersect} first shows that $\cB$ is
	$t$-intersecting.  
	Each $A^-\in\cB'$ has size at least $t$.  
	If $B\in\cB^0$, then $x_{i,m}\notin B$, implying $|A^-\cap B|=|A\cap B|\ge t$.  
	If $A,A'\in\cA$ are distinct, then $x_{i,m}\in A\cap A'$ and $|A^-\cap (A')^-|=|A\cap A'|-1\ge t$ because the hypothesis ensures $|A\cap A'|\ge t+1$.  
	Then $\cB'$ is $t$-intersecting.

Fix $A\in\cB^1$ with $w(A)>0$. 
Pick a set $F\in\cF$ assigned to $A$. 
We first claim that
\begin{equation}\label{eq:exact-prefix}
	F\cap Q_i(m)=A_i.
\end{equation}
The inclusion $A_i\subseteq F\cap Q_i(m)$ is immediate. 
If the reverse inclusion failed,  then choose
$x_{i,j}\in(F\cap Q_i(m-1))\setminus A_i$ and let
\[
\widetilde A=(A\setminus\{x_{i,m}\})\cup\{x_{i,j}\}.
\]
By $\widetilde A\preceq A\in\cB$ and \eqref{eq:left-closure-property}, we have $C\subseteq\widetilde A$  for some $C\in\cB$. 
It follows from $x_{i,m}\notin\widetilde A$ that $C\in\cB^0$ and $C\subseteq\widetilde A\subseteq F$. 
Hence $F$ is assigned to a member of $\cB^0$, a contradiction. 
This proves \eqref{eq:exact-prefix}.

We may suppose that $F\in\cH(\bs r)$ where $\bs r=(r_1,\dots,r_p)$, and write $E=F\setminus X_i$. 
The \emph{old cell} determined by $(A,\bs r,E)$ is
\[
\cC(A,\bs r,E)=
\{G\in\cH(\bs r):G\setminus X_i=E,\ G\cap Q_i(m)=A_i\}.
\]
Every member of this cell contains $A$, and is also assigned to $A$. 
Indeed, all generators use only points of $Q_i(m)$ in $X_i$, so any earlier generator contained in one member of the cell would be contained in all of them, including $F$. 
We also know that the distinct old cells associated with $A$ form a partition of the subfamily of $\cF$ consisting of the sets assigned to $A$, and hence their sizes sum to $w(A)$.

Define the \emph{expanded cell} of  $\cC(A,\bs r,E)$  by
\[
\cC^-(A,\bs r,E)=
\{G\in\cH(\bs r):G\setminus X_i=E,\ 
G\cap Q_i(m-1)=A_i\setminus\{x_{i,m}\}\}.
\]
Suppose in addition that $A\in\cA$. 
Then $A^-\in\cB'$, and since every member of the expanded cell contains $A^-$, we have $\cC^-(A,\bs r,E)\subseteq U_R(\cB')$. 

We next show $\cC^-(A,\bs r,E)\cap U_R(\cB^0)=\varnothing$. 
Suppose for contradiction that $D\subseteq G\in\cC^-(A,\bs r,E)$ for some $D\in\cB^0$. 
Then $D_i\subseteq G\cap Q_i(m-1)=A_i\setminus\{x_{i,m}\}$. 
This together with $D\setminus X_i\subseteq G\setminus X_i=F\setminus X_i$ yields $D\subseteq F$, a contradiction to the assumption that $F$ is assigned to $A$.

We also claim that $\cC^-(A,\bs r,E)\cap\cC^-(A',\bs r',E')=\varnothing$ whenever $\cC(A,\bs r,E)$ and $\cC(A',\bs r',E')$ are nonempty old cells with $(A,\bs r,E)\neq(A',\bs r',E')$. 
Suppose for contradiction that some $G'$ belongs to this intersection. 
Then $\bs r=\bs r'$ and $E=G'\setminus X_i=E'$. We also have $A_i\setminus\{x_{i,m}\}=G'\cap Q_i(m-1)=A_i'\setminus\{x_{i,m}\}$. 
Since $A,A'\in\cB^1$, both contain $x_{i,m}$, and hence $A_i=A_i'$. 
It follows that the two old cells are the same collection, i.e., $\cC(A,\bs r,E)=\cC(A',\bs r,E)$. 
Then every member of this family is assigned both to $A$ and to $A'$, so the uniqueness of the assignment gives $A=A'$. 
Hence $(A,\bs r,E)=(A',\bs r',E')$, a contradiction.

Now let $\cJ$ be the collection of all $(A,\bs r,E)$ such that $A\in\cA$, $w(A)>0$ and $\cC(A,\bs r,E)$ is an old cell associated with $A$. We have 
\[
|U_R(\cB')|\ge|U_R(\cB^0)|+\sum_{(A,\bs r, E)\in\cJ}|\cC^-(A,\bs r,E)|.
\]
Since the sets assigned to members of $\cB^0$ form exactly $U_R(\cB^0)$, we have $|U_R(\cB^0)|=|\cF|-W$. 
To get the desired result, it suffices to show
\[
\sum_{(A,\bs r, E)\in\cJ}|\cC^-(A,\bs r,E)|\ge\left(\frac{n_i}{b_i}-1+\frac1{b_i}\right)\sum_{A\in\cA}w(A).
\]

For any $(A,\bs r,E)\in\cJ$, put $a=|A_i|$.
Since $(\cB\setminus\{A\})\cup\{A^-\}$ is not $t$-intersecting, there exists $B\in\cB\setminus\{A\}$ such that  $|A^-\cap B|\le t-1$. 
This together with $|A\cap B|\ge t$ and $A^-=A\setminus\{x_{i,m}\}$ yields $x_{i,m}\in A\cap B$ and $|A\cap B|=t$. 
By Lemma \ref{lem:critical},  we have $m\le a+b_i-1$. This together with $n_i>2b_i\ge2r_i$ and $b_i\ge b_i-a+1\ge1$ yields
\[
 \frac{n_i-m+1}{r_i-a+1}\ge\frac{n_i-a-b_i+2}{b_i-a+1}=1+\frac{n_i-2b_i+1}{b_i-a+1}\ge1+\frac{n_i-2b_i+1}{b_i}=\frac{n_i}{b_i}-1+\frac1{b_i}
\]
and
\[
|\cC^-(A,\bs r,E)|=\binom{n_i-m+1}{r_i-a+1}=\frac{n_i-m+1}{r_i-a+1}\binom{n_i-m}{r_i-a}\ge\left(\frac{n_i}{b_i}-1+\frac1{b_i}\right)|\cC(A,\bs r,E)|.
\]
Recall that the sizes of the distinct old cells associated with $A$ sum to $w(A)$. 
Then
\begin{align*}
\sum_{(A,\bs r, E)\in\cJ}|\cC^-(A,\bs r,E)|&\ge\left(\frac{n_i}{b_i}-1+\frac1{b_i}\right)\sum_{(A,\bs r, E)\in\cJ}|\cC(A,\bs r,E)|\\
&=\left(\frac{n_i}{b_i}-1+\frac1{b_i}\right)\sum_{A\in\cA,\ w(A)>0}w(A)\\
&=\left(\frac{n_i}{b_i}-1+\frac1{b_i}\right)\sum_{A\in\cA}w(A),
\end{align*}
as desired.
\end{proof}

To apply Lemma~\ref{lem:ledger}, we need a subfamily $\cA\subseteq\cB^1$ in which no two distinct members have intersection exactly $t$. 
The next lemma  shows that such a subfamily can be chosen with total weight at least $W/(t+1)$.

\begin{lemma}\label{lem:critical-graph}
	Let $\cB=L^*(\cB)$ be a $t$-intersecting generating antichain of a
	shifted family.  Suppose that $i\in[p]$ satisfies $m:=m_i(\cB)>t$. 
	Set $\cV=\{A\in\cB:x_{i,m}\in A\}$. 
	For any weight function $w:\cV\to\mathbb R_{\ge0}$, put
	$W=\sum_{A\in\cV}w(A)$.  Form a graph  on $\cV$ by joining two distinct
	generators $A,B$ exactly when $|A\cap B|=t$.  Then  this graph has an independent set $\cA$ satisfying
	\[
			\sum_{A\in\cA}w(A)\ge\frac{W}{t+1}.
	\]
\end{lemma}

\begin{proof}
	For $A\in\cV$, let
	\[
	M(A)=Q_i(m-1)\setminus A_i.
	\]
	Define a color
	$\kappa(A)\in[t+1]$ by
	\[
	\kappa(A)=
	\begin{cases}
		\min\{j\in[t]:x_{i,j}\in M(A)\},
		&\text{if }M(A)\cap Q_i(t)\ne\varnothing,\\
		t+1,&\text{if }M(A)\cap Q_i(t)=\varnothing.
	\end{cases}
	\]
	We first show that it is a proper coloring.

	Let $A,B\in\cV$ be adjacent. 
	By Lemma \ref{lem:critical}, we have $A_i\cup B_i=Q_i(m)$, and  $M(A)\cap M(B)=\varnothing$ follows. 
By $x_{i,m}\in A_i\cap B_i$, we get $|M(A)|=m-|A_i|$ and $|M(B)|=m-|B_i|$.
These together with $A_i\cup B_i=Q_i(m)$ and $M(A)\cap M(B)=\varnothing$ yield
\begin{equation}\label{eq:missing-pair}
	|M(A)\cup M(B)|=|M(A)|+|M(B)|=m-|A_i\cap B_i|\ge m-t.
\end{equation}

Since $M(A)\cap M(B)=\varnothing$, by the definition of $\kappa$, the two vertices cannot receive the same color $\ell\le t$. 
It is also  impossible that $\kappa(A)=\kappa(B)=t+1$, since otherwise  both $M(A)$ and $M(B)$ are contained in $Q_i(m-1)\setminus Q_i(t)$, implying that $|M(A)\cup M(B)|\le m-t-1$, a contradiction to \eqref{eq:missing-pair}.  
Thus $\kappa(A)\ne\kappa(B)$.

Each color class is an independent set.  
The class of largest total weight has weight at least $W/(t+1)$, proving the desired result.
\end{proof}

We now combine the preceding lemmas to obtain the following localization result for generating antichains.

\begin{lemma}\label{lem:support}
	Suppose that $n_i\ge(t+2)b_i$ for every $i\in[p]$.
	Every shifted maximum $t$-intersecting subfamily of $\cH_R$ has a $t$-intersecting generating antichain $\cB$ such that
	\begin{equation}\label{eq:support}
		\cB\subseteq2^Y,\qquad
		Y=\bigcup_{i=1}^pQ_i(t).
	\end{equation}
\end{lemma}

\begin{proof}
	Let $\cF$ be a shifted maximum $t$-intersecting subfamily of $\cH_R$ with a generating family $\cD$.  
	By Lemma \ref{lem:reduced-left-closure}, $\cF$ has a generating antichain $L^*(\cD)$. 
	Since $L^*(\cD)\subseteq L(L^*(\cD))\subseteq L(\cD)$, each member of $L(L^*(\cD))$ contains a member of $L^*(\cD)$, and it follows from the fact that $L^*(\cD)$ is an antichain that $L^*(L^*(\cD))=L^*(\cD)$, i.e., $L^*(\cD)$ is reduced left-closed.

Choose a reduced left-closed generating antichain $\cB$ of $\cF$ minimizing $\sum_i m_i(\cB)$. 
Among all such choices, choose one such that the sequence $(|\cB\cap\binom{X}{|X|}|,|\cB\cap\binom{X}{|X|-1}|,\dots,|\cB\cap\binom{X}{1}|,|\cB\cap\binom{X}{0}|)$ is lexicographically minimal.

For every $i\in[p]$, $n_i\ge(t+2)b_i$ gives $n_i\ge2b_i$. 
Lemma~\ref{lem:generators-intersect} shows that $\cB$ is $t$-intersecting.  
We establish the following deletion claim.

\begin{claim}  
	For $i\in[p]$, if $m_i(\cB)>0$ and  $A\in\cB$ contains $x_{i,m_i(\cB)}$, then with $A^-=A\setminus\{x_{i,m_i(\cB)}\}$, the family $(\cB\setminus\{A\})\cup\{A^-\}$ is not $t$-intersecting.
\end{claim}
\begin{proof}
	Suppose for contradiction that $\widetilde{\cB}:=(\cB\setminus\{A\})\cup\{A^-\}$ is $t$-intersecting.  
	Then $U_R(\widetilde{\cB})$ is $t$-intersecting and contains $U_R(\cB)=\cF$, because every set containing $A$ also contains $A^-$.  
	Since $|\cF|$ is maximum, we have $U_R(\widetilde{\cB})=\cF$.

Let $\cB'=L^*(\widetilde{\cB})$. 
This is a reduced left-closed generating antichain of $\cF$.  
Left moves only decrease labels, so $m_j(\cB')\le m_j(\cB)$ for every $j\in[p]$. 
If at least one inequality is strict, then $\sum_j m_j(\cB')<\sum_j m_j(\cB)$, a contradiction to the first minimality condition imposed on $\cB$.  
We may therefore assume that every extent is unchanged.

We next compare the numbers of generators of each size in $\cB'$ with the corresponding numbers in $\cB$. 
Put $s=|A|$.  
The part of $L(\widetilde{\cB})$ coming from $A^-$ consists only of sets of size $s-1$.

Let $D\in L(\cB\setminus\{A\})$ and suppose that $D$ is not inclusion-minimal in $L(\cB)$.  
Choose  $E\in L(\cB)$ with $E\subsetneq D$.  
If $E\in L(\cB\setminus\{A\})$, then  $D$ is  not inclusion-minimal  in $L(\widetilde{\cB})$.  
Now assume $E\not\in L(\cB\setminus\{A\})$. 
We have $E\preceq A$.  
Let $e=x_{i,r}$ be the element of $E_i$ for which $r$ is largest.
 Then $E^-:=E\setminus\{e\}\preceq A^-$ and $E^-\subsetneq D$. 
Hence $E^-\in L(\{A^-\})\subseteq L(\widetilde{\cB})$, and $D$ is not inclusion-minimal in $L(\widetilde{\cB})$. 

Recall that $\cB=L^*(\cB)$. We further conclude that every member of $\cB'$ with size at least $s$ belongs to $\cB$.
Observe that $A\not\in\cB'$ by $A^-\subsetneq A$ and $A^-\in L(\widetilde{\cB})$. 
Hence the numbers of generators of sizes larger than $s$ do not increase, while the number of generators with size $s$ decreases. 
The sequence $(|\cB'\cap\binom{X}{|X|}|,|\cB'\cap\binom{X}{|X|-1}|,\dots,|\cB'\cap\binom{X}{1}|,|\cB'\cap\binom{X}{0}|)$ is therefore lexicographically smaller than the corresponding sequence for $\cB$, a contradiction to the second minimality condition imposed on $\cB$.
\end{proof}

If $\cB$ contains a $t$-set $T$, then the $t$-intersection and antichain
properties force $\cB=\{T\}$.  
Reduced left-closure forces $T\cap X_i=Q_i(|T\cap X_i|)$ for every $i$, and hence $m_i(\cB)=|T\cap X_i|\le\min\{t,b_i\}$. 
Then \eqref{eq:support} holds in this case.  
We may henceforth assume that every generator has size at least $t+1$.

Fix $i\in[p]$ with $m:=m_i(\cB)>0$. 
It remains to show that $m\le t$.  
Suppose for contradiction that $m>t$.  
Write $\cB=\cB^0\mathbin{\dot\cup}\cB^1$ and use the assignment and weights from the setup preceding Lemma~\ref{lem:ledger}.  
We have $W>0$ because, if $W=0$, then $U_R(\cB^0)=\cF$, and $L^*(\cB^0)$ would have strictly smaller $i$th extent and no larger other extent.

By Lemma~\ref{lem:critical-graph}, there exists $\cA\subseteq\cB^1$ such that no two distinct members of $\cA$ have intersection exactly $t$ and 
\[
\sum_{A\in\cA}w(A)\ge\frac{W}{t+1}.
\]
Since $n_i\ge(t+2)b_i$, we have
\[
\frac{n_i}{b_i}-1+\frac1{b_i}
\ge t+1+\frac1{b_i}>t+1.
\]
Then Lemma~\ref{lem:ledger}, together with the preceding Claim, produces  a $t$-intersecting family $\cB''=\cB^0\cup\{A^-:A\in\cA\}$ with
\[
|U_R(\cB'')|
\ge|\cF|-W+\left(\frac{n_i}{b_i}-1+\frac1{b_i}\right)\frac{W}{t+1}
>|\cF|.
\]
This contradicts the assumption that $|\cF|$ is maximal. 
Hence $m_i(\cB)\le t$ for every $i$, and  \eqref{eq:support} follows.
\end{proof}

\section{Proof of Theorem \ref{thm:multi-layer}}\label{sec:proofs}

In this section, we use the localization result of Lemma~\ref{lem:support} to prove Theorem~\ref{thm:multi-layer}. 
The following two lemmas are also needed.

Put
\[
q=\max_{i\in[p]}\frac{b_i}{n_i},\qquad
M_R=\max_{\substack{s_1+\cdots+s_p=t\\(s_1,\dots,s_p)\in\mathbb{Z}_{\ge0}^p}}\sum_{\bs r\in R}
\prod_{i=1}^p\binom{n_i-s_i}{r_i-s_i}.
\]
We remark that $M_R$ is the maximum size of a full $t$-star.

\begin{lemma}\label{lem:cylinder-vs-star}
	Let $B\subseteq X$ be an $R$-admissible set with $|B|\ge t$.  
	Then $|\cH_R[B]|\le q^{|B|-t}M_R$.
\end{lemma}
\begin{proof}
	Fix $T\in\binom Bt$. 
	Pick $\bs r=(r_1,\dots,r_p)\in R$ with $\cH(\bs r)[B]\neq\varnothing$. 
	Then $B$ and $T$ are both admissible in this layer. 
	Observe that
	\[
	|\cH(\bs r)[B]|=|\cH(\bs r)[T]|\prod_{i=1}^p\prod_{j=|T_i|}^{|B_i|-1}\frac{r_i-j}{n_i-j}\le|\cH(\bs r)[T]|\prod_{i=1}^p\left(\frac{r_i}{n_i}\right)^{|B_i|-|T_i|}\le q^{|B|-t}|\cH(\bs r)[T]|.
	\]
	Summing this inequality over $\bs r\in R$ with $\cH(\bs r)[B]\neq\varnothing$ yields the desired result.	
\end{proof}

The following direct orbit argument strengthens \cite[Lemma~2.4]{YLW}
by weakening its hypothesis.

\begin{lemma}\label{lem:unshift}
	Let $\varnothing\ne\cF\subseteq\cH_R$ be $t$-intersecting. Suppose
	$n_i>2b_i$ for each $i\in[p]$. 
	If one shift inside a part transforms $\cF$ into a full $t$-star, then
	$\cF$ itself is a full $t$-star.
\end{lemma}

\begin{proof}
	Let \[\cF^+=S_{xy}(\cF)=\cH_R[T],\] where $x,y\in X_\ell$, $|T|=t$ and the shift replaces $y$ by $x$.  
	Write $\sigma=(x\,y)$.  
	A $\sigma$-orbit is \emph{fixed} if it is a singleton, in which case its unique member contains either both $x$ and $y$ or neither.  
	Every other orbit has two members,
	which we write as $\{L,U\}$ with $x\in L$ and $y\in U$.  On such an
	orbit the shift has the following effect:
	\[
	\begin{array}{c|c}
		\cF\cap\{L,U\} & \cF^+\cap\{L,U\} \\ \hline
		\varnothing & \varnothing \\
		\{L\} & \{L\} \\
		\{U\} & \{L\} \\
		\{L,U\} & \{L,U\}.
	\end{array}
	\]
	In particular, a fixed orbit is unchanged, the empty and two-member
	choices on a two-element orbit have unique preimages, and the image of a
	shift can never contain only $U$.

	Write $t_i=|T\cap X_i|$ and let
	\[
	R_T=\{\bs r=(r_1,\ldots,r_p)\in R:t_i\le r_i\text{ for every }i\in[p]\}.
	\]
	Since $\cH_R[T]=\cF^+\ne\varnothing$, the set $R_T$ is nonempty.
	
	Suppose first that $T$ contains both or neither of $x,y$.  On every
	two-element orbit, the star $\cH_R[T]$ contains either both members or
	neither member.  The table therefore determines the preimage uniquely.
	Together with the fact that every fixed orbit is unchanged, this gives
	$\cF=\cH_R[T]$.
	
	The case $y\in T$ and $x\notin T$ is impossible.  Indeed, choose
	$\bs r\in R_T$.  In the part $X_\ell$, the set
	$X_\ell\setminus(T\cup\{x\})$ has $n_\ell-t_\ell-1$ elements, and 
	$
	n_\ell-t_\ell-1\ge r_\ell-t_\ell
	$ 
	because $r_\ell\le b_\ell$ and $n_\ell>2b_\ell$.  Thus $T$ can be
	extended to a set $U\in\cH(\bs r)$ avoiding $x$.  Then
	$U\in\cH_R[T]$, whereas $L=\sigma(U)\notin\cH_R[T]$.  Hence
	$\cF^+$ contains only the upper member $U$ of this two-element orbit,
	a contradiction to the table.
	
	It remains to consider the case $x\in T$ and $y\notin T$. 
	Put $C=T\setminus\{x\}$ and
	$Z=X\setminus(C\cup\{x,y\})$.  For each $\bs r\in R_T$, let
	\[
	\cQ_{\bs r}=\left\{P\subseteq Z:
	|P\cap X_i|=r_i-t_i\text{ for every }i\in[p]\right\}.
	\]
	For $P\in\cQ_{\bs r}$, define
	\[
	L_{\bs r}(P)=C\cup\{x\}\cup P,
	\qquad
	U_{\bs r}(P)=C\cup\{y\}\cup P.
	\]
	These are the two members of a $\sigma$-orbit in $\cH(\bs r)$, and
	$\cH_R[T]$ contains exactly $L_{\bs r}(P)$.  The table shows that
	$\cF$ contains exactly one of $L_{\bs r}(P)$ and $U_{\bs r}(P)$.
	Accordingly, define the orientation of $(\bs r,P)$ as
	\[
	\varepsilon(\bs r,P)=
	\begin{cases}
		x,&L_{\bs r}(P)\in\cF,\\
		y,&U_{\bs r}(P)\in\cF.
	\end{cases}
	\]
	We next show $\varepsilon(\bs r,P)=\varepsilon(\bs s,Q)$ for any $\bs r,\bs s\in R_T$ and $P\in\cQ_{\bs r}$,  $Q\in\cQ_{\bs s}$.

	Define a graph $\Gamma$ by
	\[
	V(\Gamma)=\{(\bs r,P):\bs r\in R_T,\ P\in\cQ_{\bs r}\},
	\]
	where two distinct vertices $(\bs r,P)$ and $(\bs s,Q)$ are adjacent
	when $P\cap Q=\varnothing$.  If adjacent vertices have different
	orientations, then $\cF$ would contain either
	$L_{\bs r}(P)$ and $U_{\bs s}(Q)$ or
	$U_{\bs r}(P)$ and $L_{\bs s}(Q)$.  In either case the two selected
	sets would intersect exactly in $C$, and hence in only $t-1$ points.
	Thus
	\[
	\varepsilon(\bs r,P)=\varepsilon(\bs s,Q)
	\quad\text{if }((\bs r,P),(\bs s,Q))\text{ is an edge of }\Gamma.
	\]
	
	We next prove that $\Gamma$ is connected.  Fix $\bs r\in R_T$ and put
	$d_i=r_i-t_i$.  The subgraph induced by
	$\{\bs r\}\times\cQ_{\bs r}$ is the direct product of the
	nontrivial Kneser graphs
	\[
	KG(N_i,d_i),
	\qquad
	N_i=
	\begin{cases}
		n_i-t_i,&i\ne\ell,\\
		n_\ell-t_\ell-1,&i=\ell,
	\end{cases}
	\]
	with the coordinates satisfying $d_i=0$ omitted.  Since
	$n_i\ge2b_i+1$, for $i\ne\ell$ we have
	\[
	N_i-2d_i
	\ge 2(b_i-r_i)+t_i+1>0,
	\]
	whereas, using $t_\ell\ge1$, we have
	\[
	N_\ell-2d_\ell
	\ge 2(b_\ell-r_\ell)+t_\ell>0.
	\]
	Every nontrivial Kneser factor is therefore connected and nonbipartite. 
	If all $d_i=0$, then $\cQ_{\bs r}=\{\varnothing\}$,
	and the corresponding subgraph consists of a single vertex and
	is connected. Otherwise, 
	Weichsel \cite{Weichsel} proved that the direct product of two graphs is connected if and only if both are connected and at least one is nonbipartite. 
	It is also well known that  the
	direct product of nonbipartite graphs is again nonbipartite. 
	Therefore, repeated application of this result shows that the subgraph induced by
	$\{\bs r\}\times\cQ_{\bs r}$ is connected.
	
	It remains to connect different layers.  Let $\bs r,\bs s\in R_T$ be distinct.
	For every $i\in[p]$,
	\[
	(r_i-t_i)+(s_i-t_i)
	\le2(b_i-t_i)
	\le n_i-t_i-1
	\le |Z\cap X_i|.
	\]
	We may
	therefore choose $P\in\cQ_{\bs r}$ and $Q\in\cQ_{\bs s}$ disjointly in
	each part.  Then $(\bs r,P)$ and $(\bs s,Q)$ are adjacent in $\Gamma$.
	Thus the layer subgraphs are joined to one another, and $\Gamma$ is
	connected.  Consequently, $\varepsilon$ is constant on $V(\Gamma)$.

	If a two-element orbit is empty in $\cF^+=\cH_R[T]$, then the table
	shows that it is also empty in $\cF$.  Moreover, since the
	shift preserves cardinality in each layer, a layer outside $R_T$
	has an empty slice in $\cF^+$ and therefore also in $\cF$.
	
	If $\varepsilon$ is constantly $x$, then every two-element orbit selected
	by the target star chooses its lower member in $\cF$.  Together with the
	fixed and empty orbits, this gives $\cF=\cH_R[T]$.
	
	Suppose instead that $\varepsilon$ is constantly $y$, and put
	\[
	T'=C\cup\{y\}=(T\setminus\{x\})\cup\{y\}.
	\]
	The sets $T$ and $T'$ have the same number of points in every part. Hence
	they are admissible in exactly the same layers.  On a fixed orbit, its
	unique member contains $T$ if and only if it contains $T'$.  
	On every two-element orbit that meets $\cH_R[T]$, the lower member contains $T$ if and only if the upper member contains $T'$.  
	Since $\cF$ contains the upper member of every two-element orbit that meets $\cH_R[T]$, while every two-element orbit disjoint from $\cH_R[T]$ remains empty, we obtain $\cF=\cH_R[T']$. 
\end{proof}

\begin{proof}[\bf Proof of Theorem~\ref{thm:multi-layer}] 
	It is sufficient to show that every maximum $t$-intersecting subfamily $\cF$ of $\cH_R$ is a full $t$-star.

	Suppose first that $\cF$ is  shifted. 
	By Lemma~\ref{lem:support}, we obtain a $t$-intersecting generating antichain
	$\cB\subseteq2^Y$, where $Y\subseteq X$ with $|Y|\le pt$.

	If each member of $\cB$ contains a common $t$-subset $T$, then $U_R(\cB)\subseteq\cH_R[T]$.   
	The maximality shows that $\cF=U_R(\cB)$ is a full $t$-star. Next we assume that $|\bigcap_{B\in\cB}B|<t$. The $t$-intersection property ensures each member of $\cB$ has size at least $t+1$.

	For $u\ge t+1$, put $\cB(u)=\cB\cap\binom Yu$. By  Lemma \ref{lem:cylinder-vs-star}, we have
	\begin{equation}\label{eq:add-proof1}
		|\cF|=|U_R(\cB)|\le \sum_{B\in\cB}|\cH_R[B]|\le M_R\sum_{u=t+1}^{|Y|}|\cB(u)|q^{u-t}.
	\end{equation}
	Katona's intersection-shadow theorem \cite{KatonaShadow} states that the $(u-t)$-shadow of the $t$-intersecting family $\cB(u)$ has size at least $|\cB(u)|$; here the shadow consists of all $(u-t)$-sets contained in at least one member of $\cB(u)$.  This shadow is contained in $\binom{Y}{u-t}$, and hence
	\[
	|\cB(u)|\le\binom{|Y|}{u-t}.
	\]
	Observe that $(1+(3pt)^{-1})^{pt}\le e^{1/3}<2$. 
	This together with \eqref{eq:add-proof1}, $q\le1/(3pt)$ and $|Y|\le pt$ yields
	\[
	|\cF|\le M_R\sum_{j=1}^{|Y|-t}\binom{|Y|}{j}q^j\le M_R\bigl((1+q)^{|Y|}-1\bigr)\le M_R\bigl((1+(3pt)^{-1})^{pt}-1\bigr)<M_R,
	\]
	a contradiction to the assumption that $|\cF|$ is maximal.

	Next suppose that $\cF$ is  non-shifted.  
	Set $\cF_0=\cF$. 
	Starting from $\cF_0$, repeatedly apply any shift inside a part that changes the current family. 
	After finitely many steps, we obtain a sequence of families $\cF_0,\cF_1,\dots,\cF_m$, where $\cF_m$ is shifted. 
	Since $|\cF_m|=|\cF_0|$, $\cF_m$ is a shifted maximum $t$-intersecting family, and therefore is a full $t$-star. 
	Recall that $n_i\ge3ptb_i>2b_i$ for every $i\in[p]$. 
	Starting from $\cF_m$, apply Lemma~\ref{lem:unshift} backwards along the finite sequence of families. 
	We finally conclude that $\cF=\cF_0$ is  a full $t$-star, as desired.
\end{proof}

\section{The nontrivial intersection problem}\label{sec:nontrivial}

In this section, we work in $\cH(\bs k)$, where $\bs k=(k_1,\dots,k_p)$ with $1\le k_i<n_i$ for every $i\in[p]$, and put $K=\sum_i k_i$.  

Recall that a set $C\subseteq X$ is admissible if $|C\cap X_i|\le k_i$ for every $i$.  
We remark here that if $n_i\ge2k_i$ for every $i$, then for each admissible set $B$ and $z\in X_i\setminus B$, since $n_i-|B\cap X_i|-1\ge k_i-|B\cap X_i|$,  $B$ can be extended to a member of $\cH(\bs k)$ avoiding $z$, and we further conclude that
\begin{equation}\label{eq:cylinder-core}
	\bigcap_{F\in\cH(\bs k)[B]}F=B.
\end{equation}

We will also repeatedly use the following observation. 
A maximum nontrivial $t$-intersecting subfamily of $\cH(\bs k)$ is inclusion-maximal among all $t$-intersecting subfamilies of $\cH(\bs k)$.
 Indeed, adjoining a new member can only decrease the common intersection, so it cannot turn a nontrivial $t$-intersecting family into a trivial one.

\subsection{Examples of nontrivial $t$-intersecting families}

In this subsection, we analyze the family $\cH^\times(T,Y)$ described in Section \ref{HMR}.

Let $\cP_t(\bs k)$ consist of all pairs $(\bs a,\bs\ell)\in\mathbb{Z}_{\ge0}^p\times\mathbb{Z}_{\ge0}^p$ satisfying
\[
0\le a_i\le k_i\quad(i\in[p]),\qquad
\sum_i a_i=t,\qquad
\sum_i\ell_i\ge2,
\]
and which, in addition, satisfy
\begin{equation}\label{eq:HM-profile-cover}
	\ell_i>0\ \Longrightarrow\ a_i<k_i
	\qquad(i\in[p]),
\end{equation}
\begin{equation}\label{eq:HM-profile-exception}
	\ell_i\le k_i-a_i+1,\quad\ell_i=k_i-a_i+1\ \Longrightarrow\ a_i>0\quad(i\in[p]),\quad|\{i\in[p]:\ell_i=k_i-a_i+1\}|\le1.
\end{equation}

Suppose $K\ge t+2$, and pick $\bs a\in\mathbb{Z}_{\ge0}^p$ with $\sum_i a_i=t$ and $0\le a_i\le k_i$ for each $i$. 
Since $K\ge t+2$, there exists $\bs\ell\in\mathbb{Z}_{\ge0}^p$ with $\sum_i\ell_i=2$ and $0\le\ell_i\le k_i-a_i$ for each $i$. 
Then $(\bs a,\bs\ell)\in\cP_t(\bs k)$ and $\cP_t(\bs k)\neq\varnothing$. 
Furthermore, for a pair $(T,Y)$ of subsets with $|Y|\ge2$ and $T\in\binom{X\setminus Y}{t}$,   its profile belongs to $\cP_t(\bs k)$ if and only if $(T,Y)$ is feasible. 
Indeed,  condition~\eqref{eq:HM-profile-cover}, together with $0\le a_i\le k_i$, is equivalent to the admissibility of every $T\cup\{y\}$ with $y\in Y$, while condition~\eqref{eq:HM-profile-exception} is equivalent to the existence of an $x\in T$ such that $(T\setminus\{x\})\cup Y$ is admissible.

\begin{proposition}
	\label{prop:product-HM}
	Let $T,Y\subseteq X$ with $|T|=t$, $|Y|\ge2$ and $T\cap Y=\varnothing$. 
	Assume that $(\bs a,\bs\ell)$ is the profile of $(T,Y)$. 
	The family $\cH^\times(T,Y)$ is $t$-intersecting and has cardinality $H^\times_{\bs n,\bs k}(\bs a,\bs\ell)$.  
	Suppose in addition that $n_i\ge2k_i$ for every $i$, and that $(\bs a,\bs\ell)\in\cP_t(\bs k)$. 
	Then $\cH^\times(T,Y)$ is nontrivial and inclusion-maximal among the $t$-intersecting subfamilies of $\cH(\bs k)$.
\end{proposition}

\begin{proof}
	Let $F$ and $F'$ be two members of $\cH^\times(T,Y)$. 
	If $T\subseteq F\cap F'$, then $|F\cap F'|\ge t$; 
	if $T\subseteq F$ and $T\not\subseteq F'$, then $|F\cap F'|\ge|F'\cap T|+|F\cap Y|\ge t$; 
	if $T\not\subseteq F$ and $T\not\subseteq F'$, then $|F\cap F'|\ge|F\cap F'\cap T|+|Y|\ge(t-2)+2=t$.  
	Thus the family is $t$-intersecting.

A direct calculation shows that the number of members of $\cH^\times(T,Y)$ containing $T$ is
\[\prod_{i=1}^p\binom{n_i-a_i}{k_i-a_i}
-\prod_{i=1}^p\binom{n_i-a_i-\ell_i}{k_i-a_i},\]
	and for a fixed $x\in T\cap X_i$, the number of members containing $T\setminus\{x\}$ but not containing $T$ is 
\[
\binom{n_i-a_i-\ell_i}{k_i-a_i-\ell_i+1}
\prod_{j\ne i}
\binom{n_j-a_j-\ell_j}{k_j-a_j-\ell_j}.
\]
Then the desired formula follows.

Now assume the additional hypotheses. 
By
\eqref{eq:cylinder-core} and $|Y|\ge2$, we have
\[
\bigcap_{y\in Y}
\bigcap_{F\in\cH(\bs k)[T\cup\{y\}]}F
=\bigcap_{y\in Y}(T\cup\{y\})=T.
\]
Note that there exists $x_0\in T$ such that the set $(T\setminus\{x_0\})\cup Y$ is admissible. 
This set can be extended to a member of $\cH(\bs k)$ avoiding $x_0$. 
Furthermore, $|\bigcap_{F\in\cH^\times(T,Y)}F|\le t-1$, and $\cH^\times(T,Y)$ is therefore nontrivial.

To show the inclusion-maximality of $\cH^\times(T,Y)$, we need the following elementary extension observation. 
If $B$ is admissible and $G\in\cH(\bs k)$, then it follows from  $n_i\ge2k_i$ and
\[|X_i\setminus(B\cup G)|=n_i-|B\cap X_i|-k_i+|B\cap G\cap X_i|
\ge k_i-|B\cap X_i|\]
for each $i\in[p]$ that there exists $F\in\cH(\bs k)[B]$  with $F\cap G=B\cap G$.

Let $G\in\cH(\bs k)\setminus\cH^\times(T,Y)$ and put $r=|G\cap T|$. 
If $r\le t-2$, then choose  $y_1\in Y$ and apply the extension observation to $T\cup\{y_1\}$. 
The resulting set belongs to $\cH^\times(T,Y)$ and meets $G$ in at most $r+1<t$ points. 
If $r=t-1$, then by $G\notin\cH^\times(T,Y)$, we have $Y\nsubseteq G$. 
Choose $y_2\in Y\setminus G$ and apply the extension observation to $T\cup\{y_2\}$. 
The resulting set belongs to $\cH^\times(T,Y)$ and  meets $G$ in exactly $t-1$ points. 
Finally, if $r=t$, then $G\cap Y=\varnothing$ by $G\notin\cH^\times(T,Y)$. 
Choose $x\in T$ for which $(T\setminus\{x\})\cup Y$ is admissible, and apply the extension observation to $(T\setminus\{x\})\cup Y$. 
The resulting  set belongs to $\cH^\times(T,Y)$ and meets $G$ in exactly $t-1$ points. 
In summary, the family $\cH^\times(T,Y)\cup\{G\}$ is not $t$-intersecting. 
Consequently, $\cH^\times(T,Y)$ is 	inclusion-maximal.
\end{proof}

Let $Z$ be a $(t+2)$-subset of $X$. 
For convenience, write
\[
\cA_1^\times(Z)=\{F\in\cH(\bs k):|F\cap Z|\ge t+1\}.
\]
It is $t$-intersecting whether or not $Z$ is admissible. 
Pick any $t$-subset $T$ of $Z$. 
It is routine to check
\begin{equation}\label{eq:HM-ball-is-HM}
	\cA_1^\times(Z)=\cH^\times(T,Z\setminus T).
\end{equation}
Let  $\cI(Z)=\{Z\setminus\{z\}:z\in Z,\ Z\setminus\{z\}\text{ is admissible}\}$. 
We have
\[
\cA_1^\times(Z)=\bigcup_{C\in\cI(Z)}\cH(\bs k)[C].
\]
If $n_i\ge2k_i$ for every $i$ and $\cI(Z)\ne\varnothing$,
then 
\[
\bigcap_{F\in\cA_1^\times(Z)}F=\bigcap_{C\in\cI(Z)}C.
\]
If $Z$ is admissible, then $\cI(Z)=\binom{Z}{t+1}$, and $\cA_1^\times(Z)$ is nontrivial. 
If $Z$ is not admissible, then $\cA_1^\times(Z)$ may be empty, trivial, or nontrivial, depending on which $(t+1)$-subsets are admissible.

We now turn to determining which profiles in $\cP_t(\bs k)$ can maximize $H^\times_{\bs n,\bs k}(\bs a,\bs\ell)$. 
\begin{lemma}\label{lem:HM-profile-endpoints}
	Suppose that $n_i\ge2k_i+2$ for every $i\in[p]$. 
	If $(\bs a,\bs\ell)$ maximizes $H^\times_{\bs n,\bs k}(\bs a,\bs\ell)$ over $\cP_t(\bs k)$ and $\sum_i\ell_i\ge3$, then $(\bs a,\bs\ell)\in\cE_t(\bs k)$.	
\end{lemma}
\begin{proof}
Fix $s\in[p]$. 
For $r\in\mathbb Z_{\ge0}$, let $\bs\ell^{(r)}$ be obtained from $\bs\ell$ by replacing its $s$th coordinate with $r$. 
Write
\[
I_s=\{r\in\mathbb Z_{\ge0}:
(\bs a,\bs\ell^{(r)})\in\cP_t(\bs k)\}.
\]
Since $\ell_s\in I_s$, we have $I_s\neq\varnothing$. 
Put
\[
d_i=k_i-a_i\quad(i\in[p]),\qquad
v_s=\sum_{i\ne s}\ell_i,\qquad
N=n_s-a_s,\qquad c=n_s-k_s.
\]
Let $L_s=\max\{0,2-v_s\}$. 
For any $r\in I_s$, it follows from $r+v_s\ge 2$ and $r\ge0$ that $r\ge L_s$. 
Write
\begin{equation*}
	U_s=\begin{cases}
		0,&a_s=k_s,\\
		d_s,&\text{$a_s<k_s$ and $\ell_i=d_i+1$ for some $i\neq s$},\\
		d_s+1,&\text{$0<a_s<k_s$ and $\ell_i\neq d_i+1$ for every $i\neq s$},\\
		k_s,&\text{$a_s=0$ and $\ell_i\neq d_i+1$ for every $i\neq s$}.
	\end{cases}
\end{equation*}
The defining conditions \eqref{eq:HM-profile-cover} and \eqref{eq:HM-profile-exception} show that $r\le U_s$. 
Then $I_s\subseteq[L_s,U_s]\cap\mathbb Z$. 
Conversely, if $r\in\mathbb Z$ satisfies $L_s\le r\le U_s$, then $r+v_s\ge2$, and the definition of $U_s$ ensures that \eqref{eq:HM-profile-cover} and \eqref{eq:HM-profile-exception} hold. 
Since all other coordinates remain unchanged, we have $(\bs a,\bs\ell^{(r)})\in\cP_t(\bs k)$, and it follows that
\[
I_s=[L_s,U_s]\cap\mathbb Z.
\]

For $r\in I_s$, set
\[
M_s(r)=H^\times_{\bs n,\bs k}(\bs a,\bs\ell^{(r)}).
\]
We can write
\begin{equation*}
	M_s(r)=\gamma_0-\gamma_1\binom{N-r}{d_s}
	+\gamma_2\binom{N-r}{c}+\gamma_3\binom{N-r}{c-1},
\end{equation*}
where $\gamma_0$ is independent of $r$, and 
\begin{align*}
	\gamma_1&=\prod_{i\ne s}\binom{n_i-a_i-\ell_i}{d_i},\\
	\gamma_2&=\sum_{j\ne s}a_j
	\binom{n_j-a_j-\ell_j}{d_j-\ell_j+1}
	\prod_{i\ne s,j}\binom{n_i-a_i-\ell_i}{d_i-\ell_i},\\
	\gamma_3&=a_s\prod_{i\ne s}
	\binom{n_i-a_i-\ell_i}{d_i-\ell_i}.
\end{align*}
\begin{claim}
	If  $z\in I_s\setminus\{L_s,U_s\}$, then \[M_s(z)<\max\{M_s(L_s),M_s(U_s)\}.\] Consequently, $\ell_s\in\{L_s,U_s\}$. 
\end{claim}
\begin{proof}
	To prove the desired inequality, suppose that $z\in I_s\setminus\{L_s,U_s\}$. 
	If $d_s=0$, then $a_s=k_s$. 
	This together with \eqref{eq:HM-profile-cover} gives $\ell_s=0$ and hence $v_s\ge2-\ell_s=2$. 
	Then $L_s=U_s=0$ and $I_s=\{0\}$, a contradiction to the existence of $z$. 
	Therefore $d_s\ge1$.

	We next show that $\gamma_1>0$, $\gamma_2,\gamma_3\ge0$ and $\gamma_2+\gamma_3>0$.
	It follows from  $I_s\ne\varnothing$ that  $\ell_i\le d_i+1$ for every $i\ne s$.
	Then
\[
n_i-a_i-\ell_i
\ge n_i-a_i-d_i-1=n_i-k_i-1\ge k_i+1\ge d_i,
\]
implying that every factor in $\gamma_1$ is positive and $\gamma_1>0$.

If  $\ell_h=d_h+1$ for some $h\neq s$, then  $a_h>0$, and the corresponding summand in $\gamma_2$ is positive. 
If $a_s=0$ and $\ell_i\le d_i$ for every $i\neq s$, then, since $\sum_i a_i=t\ge1$, there exists $j\neq s$ with $a_j>0$, and the corresponding summand in $\gamma_2$ is positive. 
Finally, if $a_s>0$ and $\ell_i\le d_i$ for every $i\neq s$,  then every factor in $\gamma_3$ is positive, and hence  $\gamma_3>0$.  
We now conclude that at least one of $\gamma_2$ and $\gamma_3$ is positive. 
Observe that $\gamma_2,\gamma_3\ge0$. 
Then $\gamma_2+\gamma_3>0$. 

	Pascal's identity, with $x=N-r-1$ and $r\in[L_s,U_s-1]\cap\mathbb Z$, gives
\begin{equation}\label{eq:HM-profile-difference}
	M_s(r+1)-M_s(r)
	=\gamma_1\binom{x}{d_s-1}-\gamma_2\binom{x}{c-1}-\gamma_3\binom{x}{c-2}.
\end{equation}
The numerical hypothesis gives $c-2>d_s-1$. 
We also have $x\ge c-2$ since if $a_s=0$, then $r\le k_s$ and in fact $x\ge c-1$, and if $a_s>0$, then $r\le d_s+1$ and $x\ge c-2$. 
Thus $\binom{x}{d_s-1}>0$.

For $q\in\{c-2,c-1\}$, put
\[
R_q(x)=\frac{\binom xq}{\binom{x}{d_s-1}}.
\]
Both $R_{c-2}(x)$ and $R_{c-1}(x)$ are strictly increasing for integers $x\ge c-2$. 
Indeed, $R_{c-1}(c-2)=0<R_{c-1}(c-1)$, and whenever $x\ge q$,
\[
\frac{R_q(x+1)}{R_q(x)}=\frac{x-d_s+2}{x-q+1}>1.
\]
By $\gamma_1\binom{x}{d_s-1}>0$, the sign of \eqref{eq:HM-profile-difference} is the sign of
\[
g(x)=1-\frac{\gamma_2}{\gamma_1} R_{c-1}(x)-\frac{\gamma_3}{\gamma_1} R_{c-2}(x).
\]
Recall that $\gamma_2,\gamma_3\ge0$ and they are not both zero.  
The function $g(x)$ is strictly decreasing in $x$.  
Since $x=N-r-1$ strictly decreases as $r$ increases, $g(N-r-1)$ is strictly increasing in $r$. 
Therefore, as $r$ increases, the sign of  $M_s(r+1)-M_s(r)$ can change only from negative to positive, and the difference can vanish for at most one value of $r$. 
Then the desired inequality follows from the assumption that $z\in I_s\setminus\{L_s,U_s\}$. 

If $\ell_s\in I_s\setminus\{L_s,U_s\}$, then applying the inequality with $z=\ell_s$ gives \[M_s(\ell_s)<\max\{M_s(L_s),M_s(U_s)\}.\] 
This together with $L_s,U_s\in I_s$ yields a contradiction  to the assumption that  $(\bs a,\bs\ell)$ maximizes $H^\times_{\bs n,\bs k}(\bs a,\bs\ell)$ over $\cP_t(\bs k)$. 
Hence $\ell_s\in\{L_s,U_s\}$.
\end{proof}

To get the desired result,  next we prove that one of the following holds:
\begin{enumerate}[label=\textup{(\roman*)},leftmargin=2.5em]
	\item for every $i\in[p]$, $\ell_i\in\{0,k_i\}$ and $\ell_i\le k_i-a_i$;
	\item there is a unique index $h$ with $1\le a_h\le k_h-1$ and $\ell_h=k_h-a_h+1$, while $\ell_i\in\{0,k_i-a_i\}$ for every $i\ne h$.
\end{enumerate}
The preceding claim shows that, for every $i\in[p]$, $\ell_i\in\{L_i,U_i\}$.  
If $\ell_m=L_m>0$ for some $m\in[p]$, then $L_m=2-v_m$ and $\sum_i\ell_i=v_m+\ell_m=v_m+L_m=2$, a contradiction to the assumption that $\sum_i\ell_i\ge3$. 
Therefore $\ell_i\in\{0,U_i\}$ for every $i\in[p]$. 

Suppose first that $\ell_i\le k_i-a_i$ for every $i\in[p]$. 
If $a_i=0$, then $U_i=k_i$ and $\ell_i\in\{0,k_i\}$.
If $0<a_i<k_i$, then $U_i=k_i-a_i+1$. 
This together with $\ell_i\le k_i-a_i$ yields $\ell_i=0\in\{0,k_i\}$. 
If $a_i=k_i$, then $U_i=0$ and hence $\ell_i=0$. 
This proves (i).

Now suppose  that $\ell_h=k_h-a_h+1$ for some $h\in[p]$. 
By \eqref{eq:HM-profile-exception}, such $h$ is unique and $a_h>0$.
We also have  $\ell_i\le k_i-a_i$ whenever $i\neq h$. 
Since $\ell_h>0$, it follows from \eqref{eq:HM-profile-cover} that $a_h<k_h$. 
For $i\neq h$, if $a_i<k_i$, then $U_i=k_i-a_i$ and $\ell_i\in\{0,k_i-a_i\}$. 
If $a_i=k_i$, then $U_i=0$ and hence $\ell_i=0$.  
This proves (ii).	 
\end{proof}

\subsection{Admissible $t$-covers and an upper bound on sizes of $t$-intersecting families}

For a nonempty $t$-intersecting family $\cF$, call a set $S\subseteq X$  a \emph{$t$-cover} of $\cF$ if $|S\cap F|\ge t$ for any $F\in\cF$, and put
\[
\tau_t(\cF)=
\min\bigl\{|C|:C\text{ is admissible and is a }t\text{-cover of }\cF\bigr\}.
\]
Every member of a $t$-intersecting family is an admissible $t$-cover, and hence $t\le\tau_t(\cF)\le K$. 
Moreover, $\tau_t(\cF)=t$ exactly when $\cF$ is trivial.  
In the following, for $0\le r\le K$, write
\begin{equation*}
	S_r=\max\{|\cH(\bs k)[C]|:C\text{ is admissible and }|C|=r\}.
\end{equation*}
Recall that $n_*=\min_{i\in[p]}n_i$.

\begin{lemma}\label{lem:HM-cylinders}
	Assume $n_*>K$ and put $\alpha=K/(n_*-K)$. 	
	Let $i\in[p]$, $C\subseteq X$ and $u\in X_i\setminus C$. 
	If $C\cup\{u\}$ is admissible, then
	\begin{equation}\label{eq:HM-one-point-ratio}
		\frac1{n_i}\le
		\frac{|\cH(\bs k)[C\cup\{u\}]|}{|\cH(\bs k)[C]|}
		=\frac{k_i-|C\cap X_i|}{n_i-|C\cap X_i|}
		\le\alpha.
	\end{equation}
	Consequently $S_{r+1}\le\alpha S_r$ for $0\le r<K$.  
	Moreover, there are admissible
	sets $C^{(0)}\subseteq C^{(1)}\subseteq\cdots\subseteq C^{(K)}$ such that
	$|C^{(r)}|=r$ and $|\cH(\bs k)[C^{(r)}]|=S_r$ for every $r$.
\end{lemma}
\begin{proof}
	A direct calculation gives the equality in \eqref{eq:HM-one-point-ratio}. 
	Since $C\cup\{u\}$ is admissible, we have $|C\cap X_i|+1\le k_i$. 
	This together with $n_i-|C\cap X_i|\le n_i$ yields the lower bound.
	The upper bound follows from $k_i-|C\cap X_i|\le K$ and $n_i-|C\cap X_i|\ge n_*-K$. 
	Furthermore, assuming in addition that $|C\cup\{u\}|=r+1$ and $|\cH(\bs k)[C\cup\{u\}]|=S_{r+1}$, we obtain $S_{r+1}\le\alpha|\cH(\bs k)[C]|\le\alpha S_r$.

	 For each part, list the marginal ratios
	\[
	\rho_{i,j}=\frac{k_i-j}{n_i-j}\qquad(0\le j<k_i).
	\]	
	Each list is strictly decreasing.
Starting with $C^{(0)}=\varnothing$, construct $C^{(r+1)}$ from $C^{(r)}$ by adding a point corresponding to the largest unused marginal. 
Then $C^{(r)}\subseteq C^{(r+1)}$. 
If the selected point corresponds to $\rho_{i,j}$, then $|\cH(\bs k)[C^{(r+1)}]|=\rho_{i,j}|\cH(\bs k)[C^{(r)}]|$. 
Since $\rho_{i,j}>0$, if  $C^{(r)}$ is admissible, then $C^{(r+1)}$ is also admissible. 
The marginals used after $r$ steps are exactly the $r$ largest members of the combined lists.  
This, together with the identity
\[
\frac{|\cH(\bs k)[D]|}{|\cH(\bs k)|}
=\prod_{i=1}^p\prod_{j=0}^{|D\cap X_i|-1}\rho_{i,j}
\]
for every admissible $r$-set $D$, yields  $|\cH(\bs k)[C^{(r)}]|=S_r$. 
In summary, the sets $C^{(0)},C^{(1)},\ldots,C^{(K)}$ form the desired sequence.
\end{proof}

Put $\beta=K-t+1$. 
We next derive an upper bound for the size of a $t$-intersecting family. 

\begin{lemma}\label{lem:HM-cover-growth}
	Assume $n_*>K$ and put $\alpha=K/(n_*-K)$. 
	Let $\cF\subseteq\cH(\bs k)$ be a nonempty $t$-intersecting family, and $m=\tau_t(\cF)\ge t+1$. 
	Suppose $\beta\alpha\le1$.  
	For every nonempty $\cG\subseteq\cF$,
	\begin{equation}\label{eq:HM-cover-growth}
		|\cG|\le\binom mt \beta^{m-t}S_m.
	\end{equation}
	Let $\cR$ consist of the members of $\cF$ containing no minimum
	admissible $t$-cover. If $m<K$, then
	\begin{equation}\label{eq:HM-cover-remainder}
		|\cR|\le\binom mt \beta^{m-t+1}S_{m+1}.
	\end{equation}
	If $m=K$, then $\cR=\varnothing$.
\end{lemma}
\begin{proof}
Fix a minimum admissible $t$-cover $D$ of $\cF$.
Every member of $\cG$ contains at least $t$ points of $D$, and therefore
\begin{equation*}
	\cG=\bigcup_{A\in\binom Dt}\cG[A].
\end{equation*}

Consider a nonempty slice $\cG[C]$ with $|C|<m$. 
 The set $C$ is admissible, but it is not a $t$-cover of $\cF$.
 Hence there exists an  $F_C\in\cF$ such that $|C\cap F_C|=t-d<t$ for some $d\ge1$. 
 If $G\in\cG[C]$, then  $|G\cap F_C|\ge t$, implying that $G$ contains at least $d$ points of $F_C\setminus C$. 
  Thus
 \[
 \cG[C]\subseteq
 \bigcup_{E\in\binom{F_C\setminus C}{d}}\cG[C\cup E].
 \]
	Since $|F_C\setminus C|=K-t+d=\beta+d-1$, the number of $E\in\binom{F_C\setminus C}{d}$ is $\binom{\beta+d-1}{d}\le \beta^d$.

	We prove, by backward induction on $s<m$, that
	\[
	|\cG[C]|\le\beta^{m-s}S_m
	\]
   for every set $C$ with $|C|=s$ and $\cG[C]\ne\varnothing$.
	Suppose that one application of the decomposition produces at most $\beta^d$ prescribed sets of size $s+d$.
	If $s+d\ge m$, then repeated application of $S_{q+1}\le\alpha S_q$ gives
	\[
	|\cG[C]|\le\beta^dS_{s+d}\le\beta^d\alpha^{s+d-m}S_m=\beta^{m-s}(\beta\alpha)^{s+d-m}S_m\le\beta^{m-s}S_m.
	\]
	In particular, if $s=m-1$, then $s+d\ge m$ and hence $|\cG[C]|\le\beta S_m$. 
	If $s+d<m$, then the induction hypothesis gives 
	\[
	|\cG[C]|\le\beta^d\beta^{m-s-d}S_m=\beta^{m-s}S_m.
	\]
	Applying the estimate above with $C=A$ and $s=t$, and summing over the
	$\binom mt$ choices of $A\in\binom Dt$, we obtain \eqref{eq:HM-cover-growth}.
	
If $m=K$, then every member of $\cF$ is  a minimum admissible $t$-cover, and hence $\cR=\varnothing$. 
Assume $m<K$. 
We apply the preceding argument to $\cR$, with $m+1$ in place of $m$. 
It remains only to verify that the decomposition is available when $|C|=m$. 
By the definition of $\cR$, if $\cR[C]\ne\varnothing$, then $C$ is not a $t$-cover of $\cF$, and hence the preceding decomposition can still be applied to $\cR[C]$. 
The same backward-induction argument therefore gives
\[
|\cR[A]|\le\beta^{m+1-t}S_{m+1}
\]
for every $A\in\binom Dt$. 
Summing over the $\binom mt$ choices of $A\in\binom Dt$ proves \eqref{eq:HM-cover-remainder}.
\end{proof}

\subsection{Proof of Theorem \ref{thm:closed-HM}}

We first analyze the family of  minimum admissible $t$-covers of a $t$-intersecting family with $t$-covering number $t+1$. 
For a nonempty $t$-intersecting subfamily $\cF$ of $\cH(\bs k)$ with $\tau_t(\cF)=t+1$, put
\[
\cC^{\rm adm}_{t+1}(\cF)=
\left\{C\in\binom X{t+1}: C\text{ is admissible and is a }t\text{-cover of }\cF\right\}.
\]
\begin{proposition}
	\label{prop:cover-geometry}
	Assume $n_i\ge2k_i$ for every $i\in[p]$. 
	Let $\cF\subseteq\cH(\bs k)$ be an inclusion-maximal $t$-intersecting family with $\tau_t(\cF)=t+1$.  
	The following hold.
	\begin{enumerate}[label=\textup{(\roman*)},leftmargin=2.5em]
		\item The family $\cC^{\rm adm}_{t+1}(\cF)$ is $t$-intersecting.
		
		\item If the members of $\cC^{\rm adm}_{t+1}(\cF)$ have no common $t$-subset, then $|\cC^{\rm adm}_{t+1}(\cF)|\ge3$, and there exists a $(t+2)$-set $Z$ such that
		\[
			\cC^{\rm adm}_{t+1}(\cF)\subseteq\binom Z{t+1}
			\qquad\text{and}\qquad
			\cF=\cA_1^\times(Z).
		\]

		\item Suppose that $\cC^{\rm adm}_{t+1}(\cF)$ has size at least $2$ and its members have a common $t$-subset.  
		Then there exists a $t$-set $Q$ and a set $Y\subseteq X\setminus Q$ with $|Y|\ge2$ such that
		\begin{equation}\label{eq:cover-star}
			\cC^{\rm adm}_{t+1}(\cF)
			=\{Q\cup\{y\}:y\in Y\}.
		\end{equation}
	\end{enumerate}
\end{proposition}

\begin{proof}
	(i) If $C\in\cC^{\rm adm}_{t+1}(\cF)$, then $\cH(\bs k)[C]\subseteq\cF$ since $|C\cap F|\ge t$ for any $F\in\cF$ and $\cF$ is inclusion-maximal.

Take $C,C'\in\cC^{\rm adm}_{t+1}(\cF)$ and put $u_i=|C\cap X_i|$, $u_i'=|C'\cap X_i|$ and $c_i=|C\cap C'\cap X_i|$. 
 Outside $C\cup C'$ there are $n_i-u_i-u_i'+c_i\ge (k_i-u_i)+(k_i-u_i')$ points in part $X_i$.
We consequently extend $C$ and $C'$ to members $F,F'\in\cH(\bs k)$ with $F\cap F'=C\cap C'$.
Then $|C\cap C'|\ge t$ and the desired result follows.

(ii) By the assumption, $|\cC^{\rm adm}_{t+1}(\cF)|\ge3$ is immediate, and the family $\cC^{\rm adm}_{t+1}(\cF)$ is contained in an inclusion-maximal nontrivial $t$-intersecting subfamily of $\binom{X}{t+1}$, which is well known to consist of  all $(t+1)$-subsets contained in a fixed $(t+2)$-subset. 
 Therefore $\cC^{\rm adm}_{t+1}(\cF)\subseteq\binom Z{t+1}$ for some $Z\in\binom{X}{t+2}$.

 The absence of a common $t$-subset forces at least three distinct admissible $t$-covers $Z\setminus\{z_1\}$, $Z\setminus\{z_2\}$ and $Z\setminus\{z_3\}$.
 If  $|F\cap Z|\le t$ for some $F\in\cF$, then  $|F\cap Z|=t$ and $z_1,z_2,z_3\notin F$, which is impossible because $|Z\setminus\{z_1,z_2,z_3\}|=t-1$. 
 Hence $|F\cap Z|\ge t+1$ for every $F\in\cF$. 
The family $\cA_1^\times(Z)$ is $t$-intersecting, and the inclusion-maximality of $\cF$ gives $\cF=\cA_1^\times(Z)$.

	(iii) In this case, every minimum $t$-cover has the unique form $Q\cup\{y\}$, where $Q$ is their common $t$-subset. 
	Let $Y$ be the collection of all added points. 
	Then \eqref{eq:cover-star} follows, and we have $|Y|\ge2$ by $|\cC^{\rm adm}_{t+1}(\cF)|\ge2$.
\end{proof}

We next record a comparison that will be used repeatedly below.  
Recall that \[\Lambda_{K,t}=\max\left\{3t^2\beta^2, 2(t+1)\beta^3\right\}.\]
Assume $K\ge t+2$ and $n_*>K(1+\Lambda_{K,t})$, and put $\alpha=K/(n_*-K)$.
By Lemma~\ref{lem:HM-cylinders}, there are admissible sets $C\subseteq C'$ with $|C|=t+1$, $|C'|=t+2$, and $|\cH(\bs k)[C]|=S_{t+1}$, $|\cH(\bs k)[C']|=S_{t+2}$. 
 We also know that $\cA_1^\times(C')$ is nontrivial. 
Inclusion--exclusion and \eqref{eq:HM-one-point-ratio} give
\begin{align}\label{eq:HM-ball-lower}
	|\cA_1^\times(C')|
	=\sum_{c\in C'}|\cH(\bs k)[C'\setminus\{c\}]|
	-(t+1)S_{t+2}\ge S_{t+1}+(t+1)(\alpha^{-1}-1)S_{t+2}>S_{t+1}.
\end{align}

\begin{lemma}\label{lem:HM-star-exception}
	Assume $K\ge t+2$ and $n_*>K(1+\Lambda_{K,t})$.  
	Let $\cF\subseteq\cH(\bs k)$ be a maximum nontrivial $t$-intersecting family.  
	Suppose that the family of its minimum admissible $t$-covers is
	\begin{equation}\label{eq:HM-minimum-cover-star}
		\{T\cup\{y\}:y\in Y\},
	\end{equation}
	where $|T|=t$, $Y\subseteq X\setminus T$, and $|Y|\ge2$.
	Then every member of $\cF$ that contains $T$ meets $Y$.
\end{lemma}

\begin{proof}
	Put $\alpha=K/(n_*-K)$.  
	The hypothesis gives $\alpha<\Lambda_{K,t}^{-1}$ and $n_i\ge2k_i$ for every $i$. 	
	The definition of $t$-covers implies that every $F\in\cF$ satisfies $|F\cap T|\ge t-1$, and equality forces $Y\subseteq F$. 
	Put
\begin{align*}
	\cU_Y&=\{F\in\cH(\bs k):T\subseteq F,\ F\cap Y\ne\varnothing\},\\
	\cA&=\{A\in\cF:T\subseteq A,\ A\cap Y=\varnothing\},\\
	\cB&=\{B\in\cF:|B\cap T|=t-1\}.
\end{align*}
	Every member of $\cU_Y$ $t$-intersects every member of $\cF$ since it shares $T$ with each member of $\cF$ containing $T$, and otherwise shares $t-1$ points of $T$ together with a point of $Y$.  
	 The maximality of $\cF$ therefore gives
	$\cU_Y\subseteq\cF$, and hence
	\begin{equation*}
		\cF=\cU_Y\mathbin{\dot\cup}\cA\mathbin{\dot\cup}\cB.
	\end{equation*}
Here $\cB\ne\varnothing$ and every $B\in\cB$ contains $Y$. 
Then $2\le|Y|\le K-t+1=\beta$. 
Since $\cF$ is maximum, the family $\cA_1^\times(C')$ fixed above and \eqref{eq:HM-ball-lower} give $|\cF|\ge|\cA_1^\times(C')|>S_{t+1}$.

The members of $\cA\cup\cB$ contain no $t$-cover from \eqref{eq:HM-minimum-cover-star}.
Lemmas~\ref{lem:HM-cylinders}  and~\ref{lem:HM-cover-growth} together with $\alpha<\Lambda_{K,t}^{-1}\le(2(t+1)\beta^3)^{-1}$ yield
\begin{align*}
	|\cA|+|\cB|&\le(t+1)\beta^2S_{t+2}
	\le(t+1)\beta^2\alpha S_{t+1},                         \\
	|\cU_Y|&\ge\bigl(1-(t+1)\beta^2\alpha\bigr)S_{t+1}
	\ge\tfrac12S_{t+1}.
\end{align*}
By $\cU_Y=\bigcup_{y\in Y}\cH(\bs k)[T\cup\{y\}]$, we have
\[
|\cU_Y|\le\sum_{y\in Y}|\cH(\bs k)[T\cup\{y\}]|.
\]
	Together with $|Y|\le \beta$, this shows that some $y_0\in Y$ satisfies
\begin{equation*}
	|\cH(\bs k)[T\cup\{y_0\}]|\ge\frac{S_{t+1}}{2\beta}.
\end{equation*}

Suppose for contradiction that $\cA\ne\varnothing$. 
Observe that $(A\cap B)\setminus T\neq\varnothing$ for any $A\in\cA$ and $B\in\cB$.
 Let
\[
\Omega=\bigcup_{A\in\cA,\,B\in\cB}(A\cap B)\setminus T.
\]
Choose a part $X_r$ meeting $\Omega$ for which $n_r$ is minimum, and then choose $A_0\in\cA$, $B_0\in\cB$ and
\[
z\in\bigl((A_0\cap B_0)\setminus T\bigr)\cap X_r.
\]
It follows from $A_0\cap Y=\varnothing$ that $z\notin Y$. 
We also have $T\cup\{z\}\subseteq A_0$, and  this set is therefore admissible.  
 By \eqref{eq:HM-minimum-cover-star}, it is not a $t$-cover of $\cF$. 
  Hence there exists $B^\ast\in\cF$ such that $|B^\ast\cap(T\cup\{z\})|<t$.
 Then
\begin{equation*}
	|B^\ast\cap T|=t-1,\qquad z\notin B^\ast,\qquad Y\subseteq B^\ast.
\end{equation*}

	Let
\[
\cD(B_0)=\{D\in\cH(\bs k)[T]\setminus\cF:
D\cap(B_0\setminus T)\ne\varnothing\}.
\]
Every member of $\cH(\bs k)[T\cup\{z\}]$ disjoint from $B^\ast\setminus T$ belongs to $\cD(B_0)$ since it meets $B_0\setminus T$ at $z$, but its intersection with $B^\ast$ is $B^\ast\cap T$, which has size $t-1<t$. 
For each $w\in B^\ast\setminus T$, if $T\cup\{z,w\}$ is not admissible, then  the family $\cH(\bs k)[T\cup\{z,w\}]$ is empty; otherwise it has size  at most $\alpha|\cH(\bs k)[T\cup\{z\}]|$ by \eqref{eq:HM-one-point-ratio}. 
Since $|B^\ast\setminus T|=\beta$, we have
\begin{equation*}
	|\cD(B_0)|\ge(1-\beta\alpha)
	|\cH(\bs k)[T\cup\{z\}]|.
\end{equation*}

	For $B\in\cB$, let $u$ be the unique point of $T\setminus B$.
Since $|A_0\cap B|\ge t$, choose $v\in(A_0\setminus T)\cap B$.
There are at most $t(\beta-1)$ pairs $(u,v)$, and
\begin{equation*}
	(T\setminus\{u\})\cup Y\cup\{v\}\subseteq B.
\end{equation*}
	Fix $y_1\in Y\setminus\{y_0\}$, and suppose $v\in X_j$. 
  Since $v\in\Omega$, the choice of $r$ gives $n_j\ge n_r$.  
We next show
\begin{equation}
	|\cH(\bs k)[(T\setminus\{u\})\cup Y\cup\{v\}]|\le
	\frac{\beta-2}{n_j-K}S_{t+1}
	\le\frac{2\beta(\beta-2)}{n_r-K}
	|\cH(\bs k)[T\cup\{y_0\}]|.\label{eq:HM-fixed-branch-1}
\end{equation}
Put $E=(T\setminus\{u\})\cup\{y_0,y_1\}$. 
	This set is  admissible because $E\subseteq B$.  
Also $E\cup\{v\}\subseteq B$, and hence
\[
\cH(\bs k)[(T\setminus\{u\})\cup Y\cup\{v\}]
\subseteq\cH(\bs k)[E\cup\{v\}].
\]	
	Since $k_j-|E\cap X_j|\le K-|E|=\beta-2$, we have
\[
\frac{|\cH(\bs k)[E\cup\{v\}]|}{|\cH(\bs k)[E]|}=\frac{k_j-|E\cap X_j|}{n_j-|E\cap X_j|}
\le\frac{\beta-2}{n_j-K}.
\]
This together with $|\cH(\bs k)[E]|\le S_{t+1}$ proves the first inequality in \eqref{eq:HM-fixed-branch-1}.
The second inequality in \eqref{eq:HM-fixed-branch-1} follows from $S_{t+1}\le2\beta|\cH(\bs k)[T\cup\{y_0\}]|$ and $n_j\ge n_r$.

By  \eqref{eq:HM-one-point-ratio}, we have
\[
\frac{|\cH(\bs k)[T\cup\{y_0\}]|}
{|\cH(\bs k)[T\cup\{z\}]|}
=\frac{|\cH(\bs k)[T\cup\{y_0\}]|/|\cH(\bs k)[T]|}{|\cH(\bs k)[T\cup\{z\}]|/|\cH(\bs k)[T]|}\le\alpha n_r.
\]
This together with \eqref{eq:HM-fixed-branch-1} and
\[
\frac{n_r}{n_r-K}=1+\frac{K}{n_r-K}\le1+\alpha,
\]
yields 
\begin{equation*}
	|\cH(\bs k)[(T\setminus\{u\})\cup Y\cup\{v\}]|
	\le2\beta(\beta-2)\alpha(1+\alpha)
	|\cH(\bs k)[T\cup\{z\}]|.
\end{equation*}
Summing over the pairs $(u,v)$ gives
\begin{equation*}
	|\cB|\le2t\beta(\beta-1)(\beta-2)\alpha(1+\alpha)
	|\cH(\bs k)[T\cup\{z\}]|.
\end{equation*}
	Write $\eta=t\beta(\beta-1)(\beta-2)$. 
 By $\alpha<\Lambda_{K,t}^{-1}\le(2(t+1)\beta^3)^{-1}$, we obtain $2\eta\alpha<1$ and
\[
\beta\alpha+2\eta\alpha(1+\alpha)
<(\beta+2\eta+1)\alpha<1.
\]
	Thus $|\cD(B_0)|>|\cB|$.

	Finally consider $\cH^\times(T,B_0\setminus T)$.  
Since $B_0\in\cB$, there exist distinct $y,y'\in Y\subseteq B_0\setminus T$. 
Observe that both $T\cup\{y\}$ and $T\cup\{y'\}$ are admissible. 
	By \eqref{eq:cylinder-core}, we have
\[
\bigcap_{F\in\cH(\bs k)[T\cup\{y\}]\cup\cH(\bs k)[T\cup\{y'\}]}F=(T\cup\{y\})\cap(T\cup\{y'\})=T.
\]
Then by $|T\cap B_0|=t-1$ and $B_0\in\cH^\times(T,B_0\setminus T)$, we know $\cH^\times(T,B_0\setminus T)$ is a nontrivial $t$-intersecting family. 
By $Y\subseteq B_0\setminus T$, every member of $\cU_Y$ belongs to $\cH^\times(T,B_0\setminus T)$. 
 Every $A\in\cA$ also belongs to $\cH^\times(T,B_0\setminus T)$, since $T\subseteq A$ and $A\cap(B_0\setminus T)\ne\varnothing$. 
  We further conclude 
 \[
 \cF\setminus\cH^\times(T,B_0\setminus T)
 \subseteq\cB\setminus\{B_0\},
 \qquad
 \cD(B_0)\subseteq\cH^\times(T,B_0\setminus T)\setminus\cF.
 \]
Therefore
\[
|\cH^\times(T,B_0\setminus T)|-|\cF|\ge |\cD(B_0)|-|\cB\setminus\{B_0\}|>0.
\]
This contradicts the maximality of $\cF$ and proves $\cA=\varnothing$.
\end{proof}

\begin{proof}[\bf Proof of Theorem~\ref{thm:closed-HM}]
	Put $\alpha=K/(n_*-K)$.
The hypothesis  $n_*>K(1+\Lambda_{K,t})$ gives $\alpha<\Lambda_{K,t}^{-1}$. 
In particular, $\beta\alpha<1$, $n_i\ge2K$, and $n_i\ge2k_i+2$ for every $i$.

For every $(\bs a,\bs\ell)\in\cP_t(\bs k)$, the condition \eqref{eq:HM-profile-exception} gives $a_i+\ell_i\le k_i+1$ for every $i\in[p]$. 
Since $n_i\ge2K\ge k_i+1$, we may choose disjoint sets $T,Y\subseteq X$ with
\[
|T\cap X_i|=a_i,\qquad |Y\cap X_i|=\ell_i
\qquad(i\in[p]).
\]
Conditions~\eqref{eq:HM-profile-cover} and \eqref{eq:HM-profile-exception} ensure that $(T,Y)$ is feasible.
Proposition~\ref{prop:product-HM} therefore gives a nontrivial $t$-intersecting family $\cH^\times(T,Y)$ of size $H^\times_{\bs n,\bs k}(\bs a,\bs\ell)$.

	Let $\cF$ be a maximum nontrivial $t$-intersecting family and  $m=\tau_t(\cF)$. 
 The nontriviality of $\cF$ gives $m\ge t+1$. 
Suppose first that $m\ge t+2$.  
For $t\le r\le K$, write
\[
f_r=\binom rt \beta^{r-t}S_r.
\]
	For $t+2\le r\le K-1$, Lemma~\ref{lem:HM-cylinders} and $\alpha<(3t^2\beta^2)^{-1}$ give
\[
\frac{f_{r+1}}{f_r}
\le\frac{r+1}{r+1-t}\beta\alpha
\le\frac{t+3}{3}\beta\alpha
<\frac{t+3}{9t^2\beta}<1.
\]
We also have
\[
\binom{t+2}{2}\beta^2S_{t+2}
\le3t^2\beta^2\alpha S_{t+1}<S_{t+1}.
\]
These together with Lemma~\ref{lem:HM-cover-growth} and \eqref{eq:HM-ball-lower} yield
\begin{equation*}
	S_{t+1}<|\cF|\le f_m\le f_{t+2}
	=\binom{t+2}{2}\beta^2S_{t+2}<S_{t+1},
\end{equation*}
a contradiction. 
Thus $m=t+1$.

	Let $\cT$ be the family of minimum admissible $t$-covers. 
We next show  $|\cT|\ge2$.	
Suppose for contradiction that $|\cT|=1$ and $\cT=\{C\}$.
It follows from Lemma \ref{lem:HM-cover-growth} that $|\cF\setminus\cF[C]|\le(t+1)\beta^2S_{t+2}$. 
This together with \eqref{eq:HM-ball-lower} yields 
\begin{equation*}
	S_{t+1}+(t+1)(\alpha^{-1}-1)S_{t+2}\le|\cF|=|\cF[C]|+|\cF\setminus\cF[C]|\le S_{t+1}+(t+1)\beta^2S_{t+2}.
\end{equation*}
Hence $\alpha^{-1}-1\le\beta^2$. 
 This contradicts $\alpha^{-1}-1>\beta^2$ from $\alpha<\Lambda_{K,t}^{-1}\le(3t^2\beta^2)^{-1}$.

	Proposition~\ref{prop:cover-geometry} applies because $\cF$ is inclusion-maximal and $n_i\ge2k_i$ for every $i$.  
	 It says that $\cT$ is a $t$-intersecting family of $(t+1)$-sets.
If its members have no common $t$-subset, then $\cF=\cA_1^\times(Z)$ for some $Z\in\binom{X}{t+2}$. 
In view of Proposition \ref{prop:cover-geometry} (ii), we may choose three different admissible $t$-covers $Z\setminus\{z_1\},Z\setminus\{z_2\},Z\setminus\{z_3\}$. 
Put $Y=\{z_1,z_2\}$ and $T=Z\setminus Y$.
Then
\eqref{eq:HM-ball-is-HM} gives
\[
\cF=\cA_1^\times(Z)=\cH^\times(T,Y).
\]
The three displayed $t$-covers show that $T\cup\{y\}$ is admissible for every $y\in Y$. 
On the other hand, since $Z\setminus\{z_3\}=(T\setminus\{z_3\})\cup Y$ is admissible,  the pair $(T,Y)$ is feasible. 
	Thus the corresponding profile belongs to $\cP_t(\bs k)$.

	It remains to treat the case in which every member of $\cT$ contains a fixed $t$-set $T$.  
Proposition~\ref{prop:cover-geometry} gives a set $Y\subseteq X\setminus T$, with $|Y|\ge2$, such that the family of minimum admissible $t$-covers is
\[
\cT=\{T\cup\{y\}:y\in Y\}.
\]
By Lemma~\ref{lem:HM-star-exception}, every member of $\cF$ containing $T$ meets $Y$. 
It follows that \[\cF\subseteq\cH^\times(T,Y).\]
 If a member $B\in\cF$ does not contain $T$, the minimum cover condition gives $|B\cap T|=t-1$ and $Y\subseteq B$. 
 The nontriviality of $\cF$ ensures that such a member exists. 
 Hence the pair $(T,Y)$ is feasible. 
  Proposition~\ref{prop:product-HM} shows that $\cH^\times(T,Y)$ is a nontrivial $t$-intersecting family, and the maximality of $\cF$ gives $\cF=\cH^\times(T,Y)$.
  
  It follows from Proposition~\ref{prop:product-HM} and the maximality of $\cF$ that the profile of $(T,Y)$ maximizes $H^\times_{\bs n,\bs k}(\bs a,\bs\ell)$ over $\cP_t(\bs k)$. 
  By the definition of $\cE_t(\bs k)$ and Lemma~\ref{lem:HM-profile-endpoints}, this profile belongs to $\cE_t(\bs k)$. 
  That is to say, the maximum size of a nontrivial $t$-intersecting subfamily of $\cH(\bs k)$ is $\max_{(\bs a,\bs\ell)\in\cE_t(\bs k)}
  H^\times_{\bs n,\bs k}(\bs a,\bs\ell)$.  
  
  In view of the preceding argument, to finish our proof, it suffices to show that, if the family $\cH^\times(T',Y')$, where $(T',Y')$ is a feasible pair and its profile lies in $\cE_t(\bs k)$, has size $\max_{(\bs a,\bs\ell)\in\cE_t(\bs k)}
  H^\times_{\bs n,\bs k}(\bs a,\bs\ell)$, then it  is a nontrivial $t$-intersecting subfamily of $\cH(\bs k)$. 
  This follows immediately from Proposition~\ref{prop:product-HM} and $\cE_t(\bs k)\subseteq\cP_t(\bs k)$.
\end{proof}

We finish with the two boundary values excluded from Theorem~\ref{thm:closed-HM}.

\begin{proposition}\label{prop:HM-boundary}
	Suppose $1\le k_i<n_i$ for every $i\in[p]$. 
	If $K=t$, then no nontrivial $t$-intersecting subfamily of $\cH(\bs k)$
	exists. 
	If $K=t+1$, then a nontrivial $t$-intersecting subfamily of $\cH(\bs k)$ exists if and only if $k_h\ge2$ for some $h\in[p]$. In this case, every nonempty nontrivial
	$t$-intersecting subfamily $\cF$ of $\cH(\bs k)$ satisfies
	\begin{equation}\label{eq:add-proof2}
	|\cF|
	\le
	\max_{(\bs a,\bs\ell)\in\cE_t(\bs k)}
	H^\times_{\bs n,\bs k}(\bs a,\bs\ell),
	\end{equation}
	and equality holds if and only if
	$\cF=\cH^\times(T,Y)$, where $|Y|=2$ and $(T,Y)$ is a feasible pair whose profile is a maximizer in $\cE_t(\bs k)$.
\end{proposition}

\begin{proof}
	If $K=t$, every two distinct members of $\cH(\bs k)$ meet in fewer than
	$t$ points.  Hence a $t$-intersecting family has at most one member and
	is trivial.
	
	Now let $K=t+1$.  
	Suppose first that $k_h\ge2$ for some fixed $h\in[p]$. 
	We next construct a nontrivial $t$-intersecting subfamily of $\cH(\bs k)$. 
	Since $n_h>k_h$, choose a $(k_h+1)$-set $V\subseteq X_h$ and three distinct points
	$u,v,w\in V$. For every $i\ne h$, fix a $k_i$-set
	$D_i\subseteq X_i$, and put
	\[
	D=\bigcup_{i\ne h}D_i.
	\]
	The three sets
	\[
	F_u=D\cup(V\setminus\{u\}),\qquad
	F_v=D\cup(V\setminus\{v\}),\qquad
	F_w=D\cup(V\setminus\{w\})
	\]
	belong to $\cH(\bs k)$. 
	Any two of them have intersection of size $|D|+k_h-1=K-1=t$, whereas $|F_u\cap F_v\cap F_w|=|D|+k_h-2=K-2=t-1$. 
	Thus $\{F_u,F_v,F_w\}$ is a nontrivial $t$-intersecting family.

	Suppose there exists a nonempty nontrivial $t$-intersecting subfamily $\cF$ of $\cH(\bs k)$. 
	Then $\cF$ is a nontrivial $t$-intersecting subfamily of $\binom{X}{t+1}$, and 	$\cF\subseteq\binom Z{t+1}$ for some $(t+2)$-set $Z$. 
	Moreover, 
	\[
	\cF\subseteq\cH(\bs k)\cap\binom Z{t+1}=\cA_1^\times(Z).
	\]
	The nontriviality of $\cF$ forces $|\cF|\ge3$. 
	Then $Z$ contains at least three distinct admissible $(t+1)$-subsets. 
	Choose two of these subsets, and let $Y$ be the set of the two points of $Z$ omitted by them.
	Put $T=Z\setminus Y$. 
	The profile of $(T,Y)$ lies in $\cE_t(\bs k)$, and \eqref{eq:add-proof2} follows from \eqref{eq:HM-ball-is-HM}. 
	
		Since $K=t+1$, the feasibility of $(T,Y)$ shows that $\cH^\times(T,Y)$ is nontrivial.  
		The characterization of all equality cases follows immediately. 
		We also know that, if $k_1=\cdots=k_p=1$, then $p=t+1$ and $Z$ contains at most two $(t+1)$-admissible subsets, a contradiction. 
		Therefore $k_h\ge2$ for some $h\in[p]$.
\end{proof}

\noindent{\bf Declaration of AI usage.} The authors acknowledge the use of AI tools during the exploratory stage of this project. All mathematical arguments and proofs presented in the final manuscript were developed and rigorously verified by the authors. 
The authors take full responsibility for the content of the manuscript.

\end{document}